\documentclass[11pt]{amsart}

\usepackage[foot]{amsaddr}

\usepackage{amssymb,amsfonts,amsmath,graphicx,verbatim,eufrak,amsthm,calc}
\usepackage{color}
\usepackage{xcolor}
\usepackage{newfloat}
\usepackage{hyperref}
\usepackage{faktor}

\usepackage{tikz}
\hypersetup{
hidelinks,backref=true,pagebackref=true,hyperindex=true,
    colorlinks=true,       
    linkcolor=violet,          
    citecolor=green,        
    filecolor=magenta,      
    urlcolor=cyan           
}

\colorlet{genial}{black} 
\colorlet{genialsol}{black}

\makeatletter

\newtheoremstyle{genialnumbox}
{7pt}
{7pt}
{\normalfont}
{}
{\small\bf\sffamily\color{genial}}
{\;}
{0.25em}
{%
{\small\sffamily\color{genial}\thmname{#1}}%
{\nobreakspace\thmnumber{\@ifnotempty{#1}{}\@upn{#2}}}
\thmnote{{\nobreakspace\the\thm@notefont\sffamily\bfseries\color{black}\nobreakspace(#3)}} 
}

\newtheoremstyle{blacknumex}
{7pt}
{7pt}
{\normalfont}
{} 
{\small\bf\sffamily}
{\;}
{0.25em}
{%
{\small\sffamily\color{genial}\thmname{#1}}%
{\nobreakspace\thmnumber{\@ifnotempty{#1}{}\@upn{#2}}}
\thmnote{{\nobreakspace\the\thm@notefont\sffamily\bfseries\color{black}\nobreakspace(#3)}} 
}

\newtheoremstyle{blacknumbox} 
{7pt}
{7pt}
{\normalfont}
{}
{\small\bf\sffamily}
{\;}
{0.25em}
{%
{\small\sffamily\color{genial}\thmname{#1}}%
{\nobreakspace\thmnumber{\@ifnotempty{#1}{}\@upn{#2}}}
\thmnote{{\nobreakspace\the\thm@notefont\sffamily\bfseries\color{black}\nobreakspace(#3)}} 
}

\newtheoremstyle{genialnum}
{7pt}
{7pt}
{\normalfont}
{}
{\small\bf\sffamily\color{genial}}
{\;}
{0.25em}
{%
{\small\sffamily\color{genial}\thmname{#1}}%
{\nobreakspace\thmnumber{\@ifnotempty{#1}{}\@upn{#2}}}
\thmnote{{\nobreakspace\the\thm@notefont\sffamily\bfseries\color{black}\nobreakspace(#3)}} 
}

\makeatother

\RequirePackage[framemethod=default]{mdframed} 

\newmdenv[skipabove=7pt,
skipbelow=7pt,
rightline=false,
leftline=false,
topline=false,
bottomline=false,
backgroundcolor=black!5,
linecolor=genial,
innerleftmargin=5pt,
innerrightmargin=5pt,
innertopmargin=10pt,
leftmargin=0cm,
rightmargin=0cm,
innerbottommargin=10pt]{tBox}

\newmdenv[skipabove=7pt,
skipbelow=7pt,
rightline=false,
leftline=false,
topline=false,
bottomline=false,
backgroundcolor=genial!10,
linecolor=genial,
innerleftmargin=5pt,
innerrightmargin=5pt,
innertopmargin=5pt,
innerbottommargin=5pt,
leftmargin=0cm,
rightmargin=0cm,
linewidth=4pt]{eBox}	

\newmdenv[skipabove=7pt,
skipbelow=7pt,
rightline=false,
leftline=true,
topline=false,
bottomline=false,
linecolor=genial!50,
innerleftmargin=5pt,
innerrightmargin=5pt,
innertopmargin=5pt,
leftmargin=0cm,
rightmargin=0cm,
linewidth=4pt,
innerbottommargin=5pt]{dBox}	

\newmdenv[skipabove=7pt,
skipbelow=7pt,
rightline=false,
leftline=false,
topline=false,
bottomline=false,
linecolor=gray,
backgroundcolor=black!5,
innerleftmargin=5pt,
innerrightmargin=5pt,
innertopmargin=5pt,
leftmargin=0cm,
rightmargin=0cm,
linewidth=4pt,
innerbottommargin=5pt]{cBox}

\newmdenv[skipabove=7pt,
skipbelow=7pt,
rightline=false,
leftline=false,
topline=false,
bottomline=false,
linecolor=gray,
backgroundcolor=black!5,
innerleftmargin=5pt,
innerrightmargin=5pt,
innertopmargin=5pt,
leftmargin=0cm,
rightmargin=0cm,
linewidth=4pt,
innerbottommargin=5pt]{pBox}

\newmdenv[skipabove=7pt,
skipbelow=7pt,
rightline=false,
leftline=false,
topline=false,
bottomline=false,
linecolor=genialsol,
innerleftmargin=5pt,
innerrightmargin=5pt,
innertopmargin=0pt,
leftmargin=0cm,
rightmargin=0cm,
linewidth=4pt,
innerbottommargin=0pt]{solBox}	

\theoremstyle{genialnumbox}
\newtheorem{thm1}{Theorem}[section]
\newtheorem{ithm1}[thm1]{$\star$ THEOREM}
\newtheorem{ques1}[thm1]{Question}
\newtheorem{conj1}[thm1]{Conjecture}

\theoremstyle{blacknumex}
\newtheorem{exer}[thm1]{Exercise}
\newtheorem{exer*}[thm1]{$\ast$ Exercise}

\theoremstyle{blacknumbox}
\newtheorem{dfn1}[thm1]{Definition}

\theoremstyle{genialnum}
\newtheorem{cor1}[thm1]{Corollary}
\newtheorem{prop1}[thm1]{Proposition}
\newtheorem{lem1}[thm1]{Lemma}

\newtheorem{exm1}[thm1]{Example}

\newenvironment{theorem}{\paragraph{ } \begin{tBox}\begin{thm1}}{\end{thm1}\end{tBox}}

\newenvironment{exe*}{\paragraph{ } \begin{eBox}\begin{exer*}}{\hfill{\color{genial}
\ensuremath{\diamond\diamond\diamond}}\end{exer*}\end{eBox}}
\newenvironment{definition}{\paragraph{ } \begin{dBox}\begin{dfn1}}{\end{dfn1}\end{dBox}}	
\newenvironment{example}{\paragraph{ } \begin{exm1}}{\hfill{\tiny%
\ensuremath{\bigtriangleup\bigtriangledown\bigtriangleup}}\end{exm1}}
\newenvironment{corollary}{\paragraph{ } \begin{cBox}\begin{cor1}}{\end{cor1}\end{cBox}}	
	
\newenvironment{conj}{\paragraph{ } \begin{cBox}\begin{conj1}}{\end{conj1}\end{cBox}}	

\newenvironment{proposition}{\paragraph{ } \begin{pBox}\begin{prop1}}{\end{prop1}\end{pBox}}	
\newenvironment{lemma}{\paragraph{ } \begin{pBox}\begin{lem1}}{\end{lem1}\end{pBox}}

\newenvironment{lem*}[1]{\vspace{1ex}\noindent
{\bf Lemma* (#1).} [restatement]  \hspace{0.5em} \em }{ }

\newenvironment{thm*}[1]{\begin{cBox}
\vspace{1ex}\noindent 
{\bf Theorem* (#1).} [restatement]  \hspace{0.5em} }{\end{cBox}}

\theoremstyle{genialnum}

\newtheorem*{clm*}{Claim}

\newenvironment{sol}%
{\begin{solBox}
\par \noindent 
\scriptsize
{\bf Solution to ex:{\color{blue} \arabic{exer}}.}  {\color{red} \ \  :( } \\ }%
{\hfill {\color{blue} :) $\checkmark$} \end{solBox}}

\newcommand{\ENDEXER}{
{\expandafter\comment}
{\expandafter\endcomment}
}

\newtheorem{remark}[thm1]{Remark}

\makeatletter
\renewcommand{\@seccntformat}[1]{\llap{\textcolor{genial}{\csname the#1\endcsname}\hspace{1em}}}                    
\renewcommand{\section}{\@startsection{section}{1}{\z@}
{-4ex \@plus -1ex \@minus -.4ex}
{1ex \@plus.2ex }
{\normalfont\large\sffamily\bfseries}}
\renewcommand{\subsection}{\@startsection {subsection}{2}{\z@}
{-3ex \@plus -0.1ex \@minus -.4ex}
{0.5ex \@plus.2ex }
{\normalfont\sffamily\bfseries}}
\renewcommand{\subsubsection}{\@startsection {subsubsection}{3}{\z@}
{-2ex \@plus -0.1ex \@minus -.2ex}
{.2ex \@plus.2ex }
{\normalfont\small\sffamily\bfseries}}                        
\renewcommand\paragraph{\@startsection{paragraph}{4}{\z@}
{-2ex \@plus-.2ex \@minus .2ex}
{.1ex}
{\normalfont\small\sffamily\bfseries}}

\makeatother

\newcommand{\IP}[1]{\left\langle #1 \right\rangle}

\newcommand{\set}[1]{\left\{#1\right\}}

\newcommand{\Integer}{\mathbb{Z}}

\newcommand{\Z}{\Integer}
\newcommand{\N}{\mathbb{N}}

\newcommand{\R}{\mathbb{R}}

\newcommand{\eps}{\varepsilon}

\newcommand{\ie}{{\em i.e.\ }}
\newcommand{\eg}{{\em e.g.\ }}

\DeclareMathOperator{\E}{\mathbb{E}}     

\renewcommand{\Pr}{}
\let\Pr\relax
\DeclareMathOperator{\Pr}{\mathbb{P}}

\newcommand{\1}[1]{\mathbf{1}_{\set{ #1 } }}

\def\squareforqed{\hbox{\rlap{$\sqcap$}$\sqcup$}}
\def\qed{\ifmmode\squareforqed\else{\unskip\nobreak\hfil
\penalty50\hskip1em\null\nobreak\hfil\squareforqed
\parfillskip=0pt\finalhyphendemerits=0\endgraf}\fi}

\newcommand{\ignore}[1]{ }

\newcommand{\dist}{\mathrm{dist}}
\newcommand{\vphi}{\varphi}

\newcommand{\Ff}{\mathcal{F}}

\newcommand{\Ee}{\mathcal{E}}

\newcommand{\Tt}{\mathcal{T}}

\newcommand{\F}{\mathbb{F}}

\newcommand{\Mm}{\mathcal{M}}

\newcommand{\define}[1]{\textbf{#1}}

\newcommand{\gr}{\mathsf{g}}

\newcommand{\rad}{\mathsf{rad}}

\newcommand{\stab}{\mathsf{stab}}

\newcommand{\pref}{\mathsf{pref}}
\newcommand{\core}{\mathsf{Core}}

\newcommand{\SL}{\mathsf{SL}}

\newcommand{\Nn}{\mathcal{N}}
\newcommand{\Ii}{\mathcal{I}}

\newcommand{\Sub}{\mathsf{Sub}}
\newcommand{\Sch}{\mathsf{Sch}}

\title{Full realization of ergodic IRS entropy in $\SL_2(\Z)$ and free groups}
\author{Liran Ron-George}
\author{Ariel Yadin}
\address{Department of Mathematics, Ben-Gurion University of the Negev}
\email{ \{lirar, yadina\}@bgu.ac.il }
\thanks{We thank Y.\ Glasner, Y.\ Hartman, T. Meyerovitch for useful discussions. 
Research partially supported by 
the Israel Science Foundation, grant no.\ 954/21.
The first author 
also partially supported by the Israel Science Foundation, 
grant no.\ 1175/18.}

\begin{document}

\maketitle

\begin{abstract}
We show that any {\em a-priori} possible entropy value is realized by an ergodic IRS, in free groups
and in $\SL_2(\Z)$.  This is in stark contrast to what may happen in $\SL_n(\Z)$ for $n \geq 3$,
where only the trivial entropy values can be realized by ergodic IRSs.
\end{abstract}

\section{Introduction}

\subsection{Normal subgroups and IRSs}

Margulis' celebrated {\em Normal Subgroup Theorem} \cite{margulis} states that 
for any irreducible lattice in a semi-simple real Lie group of real rank at least $2$,
all normal subgroups are either finite or of finite index.
The go-to example here is the lattice $\SL_n(\Z)$ in the Lie group $\SL_n(\R)$,
for $n \geq 3$.  
In a way, the theorem says that for $n \geq 3$, the group $\SL_n(\Z)$ 
doesn't have an abundance of normal subgroups.

This phenomenon is now known to be much more general.
For example, one may consider the possible IRSs in such lattices:
Let $G$ be a group.  $G$ acts on the space of subgroups of $G$ by conjugation.
An \define{invariant random subgroup}, or \define{IRS} (coined in \cite{AGV14}), 
is a probability measure on the subgroups of $G$ that is invariant under the induced action.
One example is just $\delta$-measure on a normal subgroup.  So IRSs generalize normal subgroups.
Convex combinations of IRSs are also IRSs, so when considering whether 
a group has an abundance of IRSs or not, it is natural to consider only the 
extreme points in the convex set of IRSs.  These are known as the \define{ergodic IRSs}.
Margulis' Normal Subgroup Theorem can be extended to IRSs in certain lattices,
such as $\SL_n(\Z)$ for $n \geq 3$.
See \cite{Charm, StuckZimmer} for more details.

A natural question that arises is what happens in rank-$1$ lattices, and specifically what happens 
in $\SL_2(\Z)$ (which does not usually behave like a high-rank lattice).
Does this group admit ``many'' normal subgroups, in opposition to  the high-rank case? 
Or, perhaps, ``many'' (ergodic) IRSs?
One way to make this question precise is via the notion of \define{random walk entropy}.

\subsection{Random walk entropy}

By a \define{random walk} on a group $G$, we mean the process $X_t = U_1 U_2 \cdots U_t$, 
where $(U_t)_{t=1}^\infty$ are independent group elements, all with some fixed law $\mu$.
For suitable random walks, \eg if the Shanon entropy of $U_1$ is finite,  
one may define the \define{random walk entropy} 
$h(G,\mu) = \lim_{t \to \infty} \frac{H(X_t) }{ t}$, where $H(X_t)$ is the 
Shannon entropy of the (law of the) random variable $X_t$.
(See below, Section \ref{scn:entropy}, and \cite{CoverThomas} for more on Shannon entropy.)
This entropy has deep connections to the boundary theory, see \eg \cite{KV83} and references therein. 
In a certain sense, random walk entropy is a measure of the amount of information the first step of the 
random walk provides on the behavior at infinity.  
(Even more loosely, how much does the walk at infinity ``remember'' its first step.)

If we fix a subgroup $K \leq G$, we may wish to observe the random walk only on the cosets of $H$,
\ie the process $(K X_t)_t$.  In this case, the limit of $\frac{ H( KX_t) }{ t}$ is no longer guarantied to exist.
However, there is a special case for which such a sequence does converge: when $K$ is an IRS.
The limit in this case exists by the sub-additive ergodic theorem, and is a.s.\ equal to its expectation.
So for an IRS $\lambda$ one can define the entropy 
$$ h(G,\mu,\lambda) =  \lim_{t \to \infty}  \int  \frac{ H(K X_t) }{ t} d \lambda(K) . $$
(This quantity is actually the {\em Furstenberg entropy} of some $\mu$-stationary $G$-action,
see \cite{Bowen10, HT15, HY18}.)
In the case $\lambda = \delta_N$ for some normal subgroup $N \lhd G$, the notions coincide:
$h(G/N , \bar \mu) = h(G,\mu, \delta_N)$, where $\bar \mu$ is the push-forward of $\mu$
onto the quotient group $G/N$.

We use this notion of random walk entropy to precisely define what we mean by 
``an abundance of normal subgroups / IRSs''.

\begin{definition} \label{dfn:N I}
Let $G$ be a group, and let $\mu$ be a random walk on $G$ of finite entropy.
Define the following subsets of $[0,\infty)$,
\begin{align*}
\Nn (G,\mu) & = \{ h(G/N, \bar \mu) \ : \ N \lhd G \ , \ \textrm{ $\bar \mu$ is the push-forward of $\mu$ } \} \\
\Ii (G,\mu) & = \{ h(G, \mu, \lambda) \ : \ \lambda \textrm{ is an ergodic IRS on $G$ }  \}
\end{align*}
We call $\Nn(G,\mu)$ the \define{normal spectrum} and $\Ii(G,\mu)$ the \define{IRS spectrum}.
\end{definition}

\begin{remark}
Let us stress again that we only consider {\em ergodic} IRS in the definition of $\Ii(G,\mu)$.
Otherwise, it would be very easy to just obtain the whole interval $[0, h(G,\mu)]$ by taking 
convex combinations of $\delta$-measures on the trivial subgroup and the whole group.
\end{remark}

Since $K x$ is a function of $x$, basic properties of Shanon entropy 
imply that $h(G,\mu,\lambda) \leq h(G,\mu)$.
So we have the inclusions
$$ \{ 0, h(G,\mu) \} \subset \Nn(G ,\mu) \subset \Ii(G,\mu) \subset [0, h(G,\mu) ] . $$
Thus, we say that $(G,\mu)$ has:
\begin{itemize}
\item \define{simple normal spectrum} if $\{ 0, h(G,\mu) \} = \Nn(G ,\mu)$,
\item \define{simple IRS spectrum} if $\{ 0, h(G,\mu) \} = \Ii(G ,\mu)$,
\item \define{full IRS spectrum} if $[ 0, h(G,\mu) ] = \Ii(G ,\mu)$,
\item \define{full normal spectrum} if $[ 0, h(G,\mu) ] = \Nn(G ,\mu)$.
\end{itemize}

With this definition, Margulis' Normal Subgroup Theorem implies that when $n \geq 3$,
for any $\mu$ (\eg uniform measure on a finite symmetric generating set),
$\SL_n(\Z)$ has a simple normal spectrum.
In fact, the results of \cite{StuckZimmer} extend this to IRSs: for $n \geq 3$,
$\SL_n(\Z)$ has a simple IRS spectrum.

Using this framework, one immediately arrives at the question of the structure of the sets
$\Nn, \Ii$ for the group $\SL_2(\Z)$.
Are they small as in the high-rank case, or is there some opposite behavior?

\begin{conj} \label{conj:SL2}
Let $G$ be a finitely generated group that contains a free group with finite index.
Let $\mu$ be a finitely supported symmetric probability measure on $G$.
Then $(G,\mu)$ has full normal spectrum.
\end{conj}

One example of a virtually free finitely generated group is $\SL_2(\Z)$,
which contains the free group  on $d \geq 2$ generators, $\F_d$, 
as a subgroup of finite index.
Although we cannot prove Conjecture \ref{conj:SL2}, 
it motivates trying to understand the IRS spectrum of 
random walks on $\SL_2(\Z)$ and on free groups.
Indeed this was first done by Bowen \cite{Bowen10}, 
where it was shown that for the {\em simple random walk}
on the free group $\F_d$, the IRS spectrum is full.
This was later extended in \cite{HY18} to all symmetric finitely-supported random walks on $\F_d$.

It should be noted that the above results for free groups were not sufficient even for specific random walks 
on $\SL_2(\Z)$.  The main reason is perhaps as follows: If we start with some random walk $\mu$
on a group $G$, it canonically 
induces another random walk $\mu_F$ on a finite-index subgroup $F \leq G , [G:F] < \infty$
(which is known as the {\em hitting measure}, see below).
However, when moving from $\mu$ to $\mu_F$, some properties are lost, and most importantly,
even if one starts with a finitely supported $\mu$, the resulting induced random walk $\mu_F$ 
is no longer guarantied to be finitely supported.

Thus, in order to understand the IRS spectrum of $(G,\mu)$ using $(F,\mu_F)$, 
when $F \leq G$ is of finite index, two steps are required:
{\bf Step I} is to be able to tackle infinitely supported random walks on the subgroup $F$.
For {\bf Step II} there needs to be some procedure in the construction that enables a ``lifting'' of the IRSs from 
the subgroup $F$ to IRSs on the original group $G$, and this needs to be done in such a way
that there is a relation between the corresponding random walk entropies.

Let us stress here that we are restricted to entropies arising from IRSs, so that Step II is
different than the one considered in \cite{HT15} (in that paper they considered entropy arising from 
stationary actions, which is a broader category than IRSs).
Also, our ``lifting'' is different from known procedures, such as {\em co-induction} in \cite{KQ19},
as we also control how the entropy changes in order to obtain full IRS spectrum.
In other co-induction methods we do not see a straightforward way to control the entropy.
These latter methods can provide an interval of IRS entropies near $0$, as in \cite{HT15},
but the full spectrum seems out of reach for those methods.

Thus, our main contributions can be summarized as follows:
We explain and prove how the construction introduced in \cite{HY18}, 
lends itself to naturally extending the {\em intersectional IRSs} from finite-index 
subgroups to the mother group. 
This takes car of Step II mentioned above. 
For Step I, we need to extend the construction from \cite{HY18}, so that it actually holds 
for infinitely supported random walks on free groups, which is our second main result.
Together, these culminate in the following theorems.

\begin{theorem} \label{thm:IRS free group}
Let $\mu$ be an adapted symmetric random walk with finite second moment on the free group $\F_d$
over $d \geq 2$ generators.
Then $(\F_d,\mu)$ has full IRS spectrum.
\end{theorem}

\begin{theorem} \label{thm:IRS SL2}
Let $G$ be a finitely generated group that contains a copy of $\F_d, d \geq 2$ with finite index.
Let $\mu$ be a symmetric random walk with finite second moment on $G$.

Then, $(G,\mu)$ has full IRS spectrum.

Specifically, this holds for $G = \SL_2(\Z)$.
\end{theorem}

The proof of these theorems is a combination of Theorem \ref{thm:free groups}
and Section \ref{scn:finite index}.

%

\begin{remark}
One may wonder as to the optimality of the restriction on the number of moments for $\mu$
in the above theorems.  
There are examples of random walks on the same group which are recurrent and transient
when the random walks are not restricted to a finite second moment (\eg on $\Z^2$).
Also, without the second moment restriction, 
one can have two random walks on the same group 
with both positive random walk entropy and with $0$ random walk entropy 
(for example on lamp-lighter groups over $\Z^2$, see \eg \cite{GaborBook, Woess}).
So it seems that $2$ moments is a somewhat essential assumption for preserving random walk properties.

Having said that, 
we do not know currently what happens when the finite second moment assumption is removed 
in Theorem \ref{thm:IRS SL2}.
\end{remark}

\begin{remark}
Regarding positive results for the normal spectrum in the free group,
let us mention \cite{TZ19}, where it is shown that there are ``many'' normal subgroups
with different random walk entropies.  See \cite[Theorem 1.1]{TZ19} for details.
To our knowledge this is the state of the art regarding the normal spectrum on free groups,
and it is open, for example, whether $\Nn(\F_d,\mu)$ contains an interval.
\end{remark}

\section{Notation and precise statement of results}

\subsection{Random walks}

All groups considered in this paper will be countable, so for a group $G$ we denote  
a probability measure $\mu$ on $G$ by a non-negative function $\mu : G \to [0,1]$,
such that $\sum_x \mu(x) = 1$.
Given a probability measure $\mu$ on $G$, the \define{$\mu$-random walk} on $G$
is the process $(X_t)_t$ where $X_{t+1} = X_t U_{t+1}$, and $(U_t)_{t=1}^\infty$ are independent
group elements all with law $\mu$.  Equivalently, $(X_t)_t$ is the Markov chain on state space $G$,
with transition matrix given by $P(x,y) = \mu(x^{-1} y)$.
Thus, we will also call a probability measure $\mu$ on $G$ a random walk.

Suppose $G$ is a finitely generated group.
Then, fixing some finite symmetric generating set $|S| < \infty , S=S^{-1} , G = \IP{S}$, 
we have a natural metric on $G$ via the \define{Cayley graph} with respect to $S$.
This is the graph whose vertices are $G$, and edges are given by the relation $x \sim y \iff x^{-1} y \in S$.
$\dist = \dist_{S}$ denotes the graph metric, where we omit $S$ unless we want to specifically emphasize the 
generating set.  If we denote $|x| = |x|_S = \dist(x,1)$, then it is easily seen that $\dist(x,y) = |x^{-1} y|$.

A random walk $\mu$ on $G$ is said to be \define{adapted} if the support of $\mu$ 
generates $G$ (this is the same as the associated Markov chain being {\em irreducible}).
$\mu$ is \define{symmetric} if $\mu(x) = \mu(x^{-1})$ for any $x \in G$
(which guaranties the Markov chain is {\em reversible}).
$\mu$ has finite $k$-th moment if 
$$ \E [ |U_1|^k ] = \sum_x   \mu(x) |x|^k < \infty . $$
The metrics of two Cayley graphs on the same group are bi-Lipschitz with respect to one another, 
so although the precise value of the $k$-th moment $\E [ |U_1|^k ]$ may change, its finiteness
does not depend on the specific choice of Cayley graph.

\subsection{IRS}

Consider the space $\Sub(G)$ of (closed) subgroups of $G$, with the Chabauty topology
(the topology induced by pointwise convergence of functions, recall that we only deal with countable groups). 
Let $\Mm_1(\Sub(G))$ be the space of all Borel probability measures on $\Sub(G)$.
$G$ acts on $\Sub(G)$ by conjugation, this action being continuous.  Thus, $G$ naturally acts 
on $\Mm_1(\Sub(G))$ as well.
An \define{invariant random subgroup}, or \define{IRS}, 
is a Borel probability measure $\lambda \in \Mm_1(\Sub(G))$ 
that is invariant under the $G$ action.
One may think of an IRS as a random subgroup, whose law is invariant under conjugation.
An example is $\delta_N$ for a normal subgroup $N \lhd G$.
In this sense, IRSs generalize normal subgroups.

For more information on IRSs see \cite{AGV14}.

\subsection{Entropy}

\label{scn:entropy}

More details regarding properties of Shanon entropy may be found in \cite{CoverThomas}.

Let $\mu$ be a random walk on $G$.
$\mu$ has \define{finite entropy} if the {\em Shanon entropy} is finite:
$H(\mu) : = -\sum_x \mu(x) \log \mu(x) < \infty$, with the convention that $0 \log 0 = 0$.
If $X$ is a random element with law $\mu$ then $H(X) := H(\mu)$.

If $\mu$ has finite entropy then Fekete's Lemma implies that the following limit exists:
$$ h(G,\mu) = \lim_{t \to \infty} \frac{ H(X_t) }{t} , $$
where $(X_t)_t$ is the $\mu$-random walk (so $H(X_t) = H(\mu^t)$ and $\mu^t$ is $t$ convolutions of $\mu$ 
with itself).
$h(G,\mu)$ is called the \define{random walk entropy}, and has deep connections to the Poisson boundary
and bounded harmonic functions.  See \eg \cite{KV83} and \cite[Chapter 14]{LyonsPeres} and references therein.

Given a normal subgroup $N \lhd G$, one may consider the induced random walk $\bar \mu$ on the 
quotient $G/N: \bar \mu(Nx) = \sum_{y \in Nx} \mu(y)$.
It is simple to see that $h(G/N , \bar \mu) \leq h(G, \mu)$ (this just follows from the fact that 
$H(N X_t) \leq H(X_t)$ because $N x$ is a function of $x$).
Since $G/N$ is also a group, $h(G/N , \bar \mu)$ is well defined.
However, if $K \leq G$ is not a normal subgroup, it is not guarantied that
$(\frac{ H(K X_t) }{ t} )_t$ converges.

Now assume that $K \leq G$ is a random subgroup with law $\lambda$, for some IRS $\lambda$.
Then, by the subadditive ergodic theorem, for $\lambda$-a.e.\  $K$ the limit 
$$ \lim_{t \to \infty} \frac{ H(KX_t) }{ t} $$
exists, and is a.s.\ equal to its expectation
$$ h(G,\mu,\lambda) : = \lim_{t \to \infty} \int \frac{ H(K X_t) }{ t} d \lambda(K) . $$
Again, since $H(KX_t) \leq H(X_t)$ it is immediate that $h(G,\mu,\lambda) \leq h(G,\mu)$.

Note that the set of IRSs is a convex set, and with convex combinations one may obtain convex combinations 
of the different entropy values. However, it is more interesting to understand which possible entropy values 
are possible as $\lambda$ varies over the extreme points of the convex set of IRSs.
Such an extreme point is called an \define{ergodic IRS}.

Recall Definition \ref{dfn:N I}, of the normal spectrum $\Nn$ and the IRS spectrum $\Ii$:
\begin{align*}
\Nn (G,\mu) & = \{ h(G/N, \bar \mu) \ : \ N \lhd G \ , \ \textrm{ $\bar \mu$ is the push-forward of $\mu$ } \} \\
\Ii (G,\mu) & = \{ h(G, \mu, \lambda) \ : \ \lambda \textrm{ is an ergodic IRS on $G$ }  \}
\end{align*}

\subsection{Intersectional IRS}

In this section we describe a certain construction of IRSs from \cite{HY18},
called \define{intersectional IRSs}.

Let $K \leq G$ be a subgroup.
Let $N = N_G(K) = \{ g \in G \ : \ K^g = K \}$ be the {\em normalizer} of $K$ in $G$.
(Here $x^g = g^{-1} x g$ so $K^g = g^{-1} K g$.)
Let $\Theta = \Theta_G(K) = G/N$ be the set of right-cosets of $N$.
Since $K^n = K$ for any $n \in N$, the conjugation $K^{N g} : = K^g$ is well defined,
for any $N g \in \Theta$.

Given a subset $\emptyset \neq A \subset \Theta$, consider the subgroup
$\core_A(K) : = \bigcap_{\theta \in A} K^\theta$.
$G$ acts naturally from the right on $\Theta$, via $(N\gamma) .g = N\gamma g$,
so it acts on subsets of $\Theta$ as well.  It is simple to verify that
the map $A \mapsto \core_A(K)$ from subsets of $\Theta$ to $\Sub(G)$ is $G$-equivariant.
That is, $(\core_A(K) )^g = \core_{A.g} (K)$.

Now, let $A \subset \Theta$ be a random subset, chosen from the product-Bernoulli-$p$ measure;
that is, for some fixed $p \in (0,1)$, 
we have $\theta \in A$ with probability $p$, and all $\theta \in \Theta$ are independent.
Since the law of $A$ is invariant to the $G$-action (\ie $A$ and $A.g$ have the same distribution),
we get that $\core_A(K)$ is an IRS.
In \cite{HY18} it is shown that if $\Theta$ is infinite (equivalently, if $|K^G| = \infty$),
then this IRS is ergodic.  We denote this IRS by $\lambda_{p,K} \in \Mm_1(\Sub(G))$.

We now discuss 
the weak$^*$ limits of $\lambda_{p,K}$ as $p \to 0$ and $p \to 1$.

For $K \leq G$ and $N = N_G(K)$ and $\Theta = \Theta_G(K) = G/N$ as above,
define:
$$ || g ||_K = | \Theta \setminus \Omega_g | 
\qquad \qquad \Omega_g = \{ \theta \in \Theta \ | \ g \in K^\theta \} . $$
It is easily verified that since $K^\theta$ is a group, we have 
$\Omega_g \cap \Omega_\gamma \subset \Omega_{g \gamma}$ and $\Omega_g = \Omega_{g^{-1} }$.
Also, since $g^{-1} K^\theta g = K^{\theta.g}$, we have $\Omega_{g^\gamma} = \Omega_g . \gamma^{-1}$.
Thus,
\begin{itemize}
\item $|| g \gamma ||_K \leq || g ||_K + || \gamma ||_K$
\item $|| g^{-1} ||_K = || g ||_K$
\item $|| g^\gamma ||_K = || g ||_K$,
\end{itemize}
(which justifies the ``norm'' notation).
These properties give rise to two natural normal subgroups of $G$:
\begin{align*}
\core_G(K) & = \{ g \in G \ : \ || g ||_K = 0 \} = \bigcap_{\theta \in \Theta} K^\theta , \\
\core_\emptyset(K) & = \{ g \in G \ : \ || g ||_K < \infty \} .
\end{align*}
$\core_G(K)$ is just the {\em normal core} of $K$ in $G$, \ie the intersection of all conjugates of $K$.
$\core_\emptyset(K)$ is the collection of all elements $g \in G$ the appear in all but finitely many conjugates.
The use of the empty set in the notation $\core_\emptyset(K)$ is justified by the following.

For any $\emptyset \neq A \subset \Theta$, note that $g \in \core_A(K)$ if and only if
$A \subset \Omega_g$.  
So if $A \subset \Theta$ is chosen from the product-Bernoulli-$p$ measure, then
$$ \lambda_{p,K} \big( \{ H \in \Sub(G) \ : \ g \in H \} \big) = (1-p)^{|| g ||_K } , $$
where $(1-p)^\infty = 0$.
We arrive at the conclusion that the weak$^*$ limits of $\lambda_{p,K}$ are:
\begin{proposition}
In the weak$^*$ topology we have that
\begin{align*}
\lambda_{p,K} & \to \delta_{\core_\emptyset(K) } \qquad \textrm{ as } p \to 0 \\
\lambda_{p,K} & \to \delta_{\core_G(K) } \qquad \textrm{ as } p \to 1 \\
\end{align*}
\end{proposition}
With this in mind, we naturally define $\lambda_{0,K} = \delta_{\core_\emptyset(K)}$ and 
$\lambda_{1,K} = \delta_{\core_G(K)}$.

The importance of these limits is in the following,
which is Proposition 3.3 of \cite{HY18}.

\begin{proposition} \label{prop:cont entropy}
Assume that $\mu$ is a symmetric, adapted probability measure with finite entropy.
Let $h(p) = h(G,\mu,\lambda_{p,K})$.

If $h(0)=0$ then the function $h$ is continuous.

Specifically, if $h(G / \core_\emptyset(K) ,\bar \mu) = 0$, 
then $[0, h(G/ \core_G(K) , \bar \mu) ] \subset \Ii(G,\mu)$.
\end{proposition}

So, in essence, in order to show that the IRS spectrum $\Ii(G,\mu)$ is large,
we need to find a subgroup $K \leq G$, such that $\core_\emptyset(K)$ is ``large''
and $\core_G(K)$ is ``small''.
In \cite{HY18} such a subgroup $K$ was constructed using Schreier graphs of the free group.
In that construction, which we will shortly describe, it turns out that $G/ \core_\emptyset(K)$ is 
nilpotent, implying that $h(G / \core_\emptyset(K) , \bar \mu) = 0$ for any $\mu$,
by the Choquet-Deny Theorem \cite{CD, Raugi}.
The difficulty there was to show that $h(G / \core_G(K) , \bar \mu)$ is large,
and in \cite{HY18} this was only done for finitely-supported random walks on the free group.
In this paper we will show that the construction in \cite{HY18} actually works for non-finitely supported
random walks.
Specifically, we prove:

\begin{theorem} \label{thm:free groups}
Let $\mu$ be an adapted, symmetric random walk on the free group $\F_d$, with $d \geq 2$.
Assume that $\mu$ has finite second moment.

Then, there exists a non-increasing sequence $K_{n+1} \leq K_n \leq \F_d$ of subgroups
such that the following holds.
\begin{itemize}
\item For every $n$,
the group $\F_d / \core_{\emptyset}(K_n)$ is nilpotent,
and therefore $$ h(\F_d , \mu , \lambda_{0,K_n} ) = 0 . $$

\item For every $n$, the function $p \mapsto h(\F_d, \mu, \lambda_{p,K_n})$ is continuous,
and therefore
$$ [0, h(\F_d/ \core_{\F_d}(K_n)  , \bar \mu) ] \subset \Ii(\F_d, \mu) . $$

\item We have
$$ \lim_{n \to \infty} h(\F_d/ \core_{\F_d}(K_n) , \bar \mu) = h(\F_d , \mu)  $$
and therefore $\F_d$ has full IRS spectrum.
\end{itemize}
\end{theorem}

This immediately implies Theorem \ref{thm:IRS free group}.
The proof of Theorem \ref{thm:free groups} is in Section \ref{scn:Schreier}.

\subsection{Finite index subgroups}

\label{scn:finite index}

In order to prove Theorem \ref{thm:IRS SL2}, we require a way of passing from an intersectional 
IRS on $F$ to an ergodic IRS on $G$, where $[G:F] < \infty$, in a way that controls the random walk entropy.
There are some standard ways to ``lift'' IRSs from finite index subgroups, such as {\em co-induction} from \cite{KQ19},
but these do not control the entropy in a quantitative way, to enable a lifting the full IRS spectrum 
to a full IRS spectrum in the mother group.  Another possibility is to realize the {\em Furstenberg entropy 
of stationary spaces} as in \cite{HT15}, but this is a broader class of entropies than those arising from ergodic IRSs.

For this reason the explicit construction of an intersectional IRS in Theorem \ref{thm:free groups} is useful,
because it provides a random subgroup of the free group, which can also be thought of as a random subgroup
of the mother group.  We now explain how to connect the IRS entropy spectra of the two.

Let $F \leq G , [G:F]<\infty$ be a finite index subgroup.
For a random walk $\mu$ on $G$, there is an induced random walk $\mu_F$ on $F$,
defined as follows:
Let $(X_t)_t$ be the $\mu$-random walk.  Let $T = \inf \{ t \geq 1 \ : \ X_t \in F \}$ be the return time to $F$.
Define $\mu_F(x) = \Pr [ X_T = x ]$.  $\mu_F$ is sometimes called the \define{hitting measure},
because it is the probability measure of the first time the random walk returns to the subgroup $F$.

The hitting measure is well defined whenever $T< \infty$ a.s.
But when $[G:F] < \infty$, then $\mu_F$ will also inherit some properties from $\mu$.

\begin{proposition} \label{prop:hitting measure}
Let $G$ be a finitely generated group, and $F \leq G$ a subgroup of finite index $[G:F] < \infty$.
If $\mu$ is a symmetric, adapted random walk on $G$, with finite $k$-th moment,
then $\mu_F$ is a symmetric, adapted on $F$, and has finite $k$-th moment as well.
\end{proposition}

\begin{proof}
This basically follows from the proof of  \cite[Proposition 3.3]{BE95}, 
or from the proof of \cite[Lemma 3.2]{MY16}.

Let $(X_t)_t$ be the $\mu$-random walk, and let $T = \inf \{ t \geq 1 \ : \ X_t \in F \}$.
In the proof of Lemma 3.2 of \cite{MY16}, it is shown that 
$|X_T|$ is stochastically dominated by $W = \sum_{j=1}^\tau N_j$,
for some independent random variables $\tau , N_1, N_2 , \ldots$, 
such that $\tau$ has the same 
distribution as $T$, and $(N_j)_j$ take on only non-negative integer values,
and all have the same distribution.  Moreover, if $\mu$ has finite $k$-th moment, 
then $\E [N_j^k] < \infty$.

Since the process $(F X_t)_t$ is an irreducible Markov chain on the {\em finite} state space $G/F$, 
it is well known that $\Pr [ \tau > t ] \leq e^{-ct}$ for some $c>0$ and all $t>0$ (see \eg \cite{LPW}).
Specifically, $\E[\tau^k] < \infty$.
Recall that by Jensen's inequality, for any $j_1, \ldots, j_k \in \N$ (not necessarily distinct)
$$ \E [ \prod_{i=1}^k N_{j_i} ] \leq \E[N_1^k ] . $$
So using the independence of $\tau, (N_j)_j$,
\begin{align*}
\E [|X_T|^k] & \leq \E [ |W|^k ]  \leq 
\sum_t \Pr [ \tau =t ] \cdot \E [N_1^k] \cdot t^k
\leq \E [\tau^k] \cdot \E [N_1^k ] < \infty .
\end{align*}
\end{proof}

The following proposition is well known and by no means original, the proof is included for completeness.
(See \cite{Abramov} for more on this type of relation.)

\begin{proposition} \label{prop:abramov}
Let $F \leq G$ be a subgroup of finite index, $[G:F] < \infty$.
Let $\mu$ be an adapted symmetric random walk on $G$, and let $\mu_F$ be the corresponding 
hitting measure on $F$.

Then, for any $N \lhd F$ which is normal in $G$,
$$ h(G/N , \bar \mu) = [G:F] \cdot h(F/N , \overline{\mu_F} ) . $$
\end{proposition}

\begin{proof}
Let $(X_t)_t$ denote the $\mu$-random walk on $G$.
Define inductively $T_0 = 0$ and 
$$ T_{n+1} = \inf \{ t \geq T_n + 1 \ : \ X_t \in F \} . $$
Define $Y_n = X_{T_n}$, which results in $(Y_n)_n$ being a $\mu_F$-random walk on $F$.

Since $F/N, G/N$ are groups, we have that 
\begin{align*}
h(G/N , \bar \mu ) & = \lim_{t \to \infty} \frac{ H(N X_t)}{t} \\
h(F/N , \overline{\mu_H} ) & = \lim_{n \to \infty} \frac{ H(N Y_n) }{ n } 
\\
& = \lim_{n \to \infty} \frac{ H(N X_{T_n})  }{ T_n } \cdot  \lim_{n \to \infty} \frac{ T_n}{n} 
= h(G/N , \bar \mu ) \cdot \lim_{n \to \infty} \frac{ T_n}{n} ,
\end{align*}
whenever $\lim_{n \to \infty} \frac{ T_n }{ n}$ exists.

Write $D_n = T_n- T_{n-1}$ for $n \geq 1$, and note that $(D_n)_n$ form a sequence
of independent random variable, all with the same distribution, that of $T_1$.
It is well known (see \eg \cite{LPW} or \cite{Abramov}) that $\E[T_1] = [G:F]$.
Thus $\lim_{n \to \infty} \frac{ T_n }{ n} = [G:F]$ a.s., by the law or large numbers, completing the proof.
\end{proof}

The following extends the results of Theorem \ref{thm:free groups} to $G$ which is virtually free,
proving Theorem \ref{thm:IRS SL2}.

For the rest of this section we assume that $G$ is a finitely generated group and $F \leq G$
is a subgroup of finite index $[G:F] < \infty$ such that $F \cong \F_d , d \geq 2$ 
is a finitely generated free group.
Since subgroups of free groups are free as well, by passing to a subgroup of $F$ of finite index,
we may assume that $F$ is a normal subgroup of $G$.
Let $\mu$ be an adapted, symmetric random walk on $G$, with finite second moment.
Let $\mu_F$ be the hitting measure on $F$.

Let $K \leq F$.  We may regard $K$ as a subgroup of $G$ and of $F$.
So $K$ will induce intersectional IRSs in both $G$ and $F$.
In order to differentiate between $\core_\emptyset(K)$ when thinking of $K$ as a subgroup of $G$
and when thinking of $K$ as a subgroup of $F$, we will put a superscript
$\core_\emptyset^G(K) , \core_\emptyset^F(K)$.  Similarly for the norm $|| g ||_K^G , || g ||_K^F$.

\begin{lemma} \label{lem:nilpotent for empty core}
If $F / \core_\emptyset^F(K)$ is nilpotent, then $G / \core_\emptyset^G(K)$ is virtually nilpotent.
\end{lemma}

\begin{proof}
Denote $C_G = \core_\emptyset^G(K)$ and $C_F = \core_\emptyset^F(K)$.
Note that, 
$$ G / C_G \cong \faktor{G / (F \cap C_G) }{ C_G / (F \cap C_G) } . $$
So it suffices to show that $G / (F \cap C_G)$ is virtually nilpotent.
Since
$$ [ G / (F \cap C_G) : F  / (F \cap C_G) ] = [G:F] < \infty , $$
so it suffices to show that $F / (F \cap C_G)$ is nilpotent.

Let $T$ be a set of representatives of left-cosets of $F$ in $G$.  
So $|T|=[G:F]$ and $G = \biguplus_{t \in T } t F$.
Note that for any fixed $\vphi \in F$,
$$ \{ K^g \ : \ g \in G \ , \ \vphi \not\in K^g \} \subseteq \bigcup_{t \in T} \{ K^{tf} \ : \ f \in F \ , \ \vphi \not\in K^{tf} \} $$
which implies that (considering $K^t$ as subgroups of $F$) we have 
$$  || \vphi ||_{K}^G \leq \sum_{t \in T} || \vphi ||_{K^t}^F . $$
Thus,
$$ C : = \bigcap_{t \in T} \core_{\emptyset}^F (K^t) \subseteq F \cap C_G . $$

We assumed that $F / C_F$ is nilpotent.
Let $F^{(0)} = F , F^{(k+1)} = [F^{(k)} , F]$ be the lower central series of $F$.
So for some $n$ we have $F^{(n)} \lhd C_F$.
Note that since $F \lhd G$, we have that  $K^g \leq F$ for all $g \in G$.
So, for any fixed $t \in T$,
\begin{align*}
|| \vphi^t ||_{K^t}^F & = \# \{ K^{tf} \ : \ f \in F \ , \ \vphi^t \not\in K^{tf} \} 
= \# \{ K^{ft} \ : \ f \in F \ , \ \vphi \not\in K^f \} = || \vphi ||_K^F .
\end{align*}
Thus, the isomorphism $\vphi \mapsto \vphi^t$, which is an automorphism of $F$
(since $F \lhd G$), maps $\core_\emptyset^F(K)$ onto $\core_\emptyset^F(K^t)$.
Since $F^{(n)}$ is preserved by any automorphism, we conclude that 
$F^{(n)} \lhd \core_\emptyset^F(K^t)$ for all $t \in T$.
This implies that $F / C$ is nilpotent as well.

Finally, 
$$ F / (F \cap C_G) \cong \faktor{ F / C }{ (F \cap C_G) / C }  $$
implying that $F / (F \cap C_G)$ is nilpotent, as required.
\end{proof}

\begin{lemma} \label{lem:entropy for core}
We have that
$$ h( G / \core_G(K) , \bar \mu ) \geq [G:F] \cdot h( F / \core_F(K) , \overline{\mu_F} ) . $$
\end{lemma}

\begin{proof}
Denote $C_G = \core_G (K)$ and $C_F = \core_F(K)$.
First, note that 
$$ C_G = \bigcap_{g \in G} K^g \subset \bigcap_{g \in F} K^g = C_F . $$
This implies that $H( C_F X_t) \leq H(C_G X_t)$
for any $t$.  Thus, $h(G/ C_G , \bar \mu) \geq h(G / C_F , \bar \mu)$.

By Proposition \ref{prop:abramov} we have that 
$$ h(G/C_F , \bar \mu) = [G:F] \cdot h(F / C_F , \overline{\mu_F}) $$
completing the proof.
\end{proof}

Combining Theorem \ref{thm:free groups} with Lemmas \ref{lem:nilpotent for empty core}
and \ref{lem:entropy for core} we can prove Theorem \ref{thm:IRS SL2}.

\begin{proof}[Proof of Theorem \ref{thm:IRS SL2}]
Let $F \leq G, [G:F] < \infty$ such that $F \cong \F_d$ is a finitely generated free group ($d \geq 2$).
By Proposition \ref{prop:hitting measure}, given an adapted, symmetric random walk $\mu$ on $G$,
with finite second moment, the hitting measure $\mu_F$ is a symmetric, adapted random walk on $F$
with finite second moment, so Theorem \ref{thm:free groups} is applicable to $(F,\mu_F)$.

Let $K_{n+1} \leq K_n \leq F$ be the subgroups as guaranteed by Theorem \ref{thm:free groups}.

Theorem \ref{thm:free groups} tells us that $F / \core_\emptyset^F(K_n)$ is nilpotent,
so that by Lemma \ref{lem:nilpotent for empty core} we have that $G / \core_\emptyset^G(K_n)$
is virtually nilpotent.  The Choquet-Deny Theorem \cite{CD, Raugi} tells us that 
$h(G / \core_\emptyset^G(K_n) , \bar \mu ) = 0$.
So by Proposition \ref{prop:cont entropy} 
we have that $[0, h(G / \core_G(K_n)  , \bar \mu) ] \subset \Ii(G,\mu)$.

By Lemma \ref{lem:entropy for core} and Theorem \ref{thm:free groups} we have that
$$ h( G / \core_G(K_n) , \bar \mu ) \geq [G:F] \cdot h( F / \core_F(K_n) , \overline{\mu_F} ) 
\to [G:F] \cdot h(F , \mu_F)  $$
as $n \to \infty$.
By Proposition \ref{prop:abramov} we have $h(G,\mu) = [G:F] \cdot h(F, \mu_F)$.
We conclude that $[0, h(G,\mu)] = \Ii(G,\mu)$.
\end{proof}

\section{Intersectional IRSs in the free group}

\label{scn:Schreier}

In this section we prove Theorem \ref{thm:free groups}.

\subsection{Schreier graphs of free groups}

We begin with some notation and simple facts regarding Schreier graphs.

Let $S$ be a finite set.
A \define{rooted, $S$-labeled, oriented multigraph} is a tuple 
$(V,S,\ell,o)$ with the following properties: $V$ is a non-empty set, 
$\ell : S \times V \to V$, and $o \in V$.
We think of $( v, \ell(s,v) )$ is an oriented edge in the graph,
labeled by $s$.  
$o$ is the root of the graph.
Note that multiple edges and self-loops are allowed.
Also, at every vertex there is exactly one edge labeled by each $s \in S$, 
oriented outwards.
Such a graph is called \define{proper} if for every $u \in V$ there is $v \in V$ such that $\ell(s,v)=u$;
that is, at every vertex there is also an incoming oriented edge labeled by $s$.
The edge with reversed orientation $(\ell(s,v) , v )$ is thought to be labeled
by $s^{-1}$ (which at the moment is just a new label).
We write $\ell(s^{-1},u) = v$ in this case.

Note that by forgetting the orientation, this induces a graph structure on the vertex set $V$,
by letting $\{v,u\}$ be an edge in the graph whenever $\ell(s,v) = u$ or $\ell(s,u)=v$ for some $s \in S$.
So there is a natural graph metric, and a notion of paths in such cases.

An isomorphism between two proper, rooted, labeled, oriented multigraphs
$(V,S,\ell,o)$ and $(V',S',\ell',o')$,
is a bijection $\vphi : V \to V'$ such that $|S| = |S'|$,
for every $s \in S$ there is $s' \in S'$ such that $\ell'(s',\vphi(v)) = \vphi(\ell(s, \vphi(v)))$
for all $v \in V$, and such that $\vphi(o) = o'$.
\ie a graph isomorphism preserving the labeling and orientation.

One can generate such graphs using group actions.
Let $G$ be a finitely generated group, and fix some finite, symmetric, 
generating set $S = S^{-1} , |S| < \infty, G = \IP{S}$, and
assume that no element of $S$ is its own inverse (this is a technical condition, in order to make the presentation simpler).
If $G$ acts from the right on some set $X$, then for a root $o \in X$ we may define
$\ell(s,x) = x.s$ for all $s \in S$, and we obtain the {\em Schreier graph} of the $G$-action on $X$.
Moreover, taking $K = \stab(o)$ one obtains a subgroup $K \leq G$. 

The other direction is also possible.
For a subgroup $K \leq G$, the \define{Schreier graph} of $K$ (with respect to $S$),
denoted $\Sch(K) = \Sch_S(K)$, is the proper, rooted, $S$-labeled, oriented multigraph
$(V,S,\ell,o)$ with $V = G/K , \ell(s,Kg) = Kgs$ and $o = K$.
We denote by $|Kg| = |Kg|_S$ the graph distance in $\Sch_S(K)$ between $Kg$ and the root $K$.
Also, $B(Kg,r) = B_S(Kg,r)$ denotes the ball of radius $r$ around $Kg$ in the graph $\Sch_S(K)$.

Note that if $K \lhd G$ is a normal subgroup, then the Schreier graph $\Sch_S(K)$
is the Cayley graph of $G/K$ with respect to the generating set $\{ Ks \ : \ s \in S \}$.

One advantage of the free group is in the ease of constructing Schreier graphs, and hence subgroups.
Suppose that $|S|=d$ and consider a proper, rooted, $S$-labeled, oriented multigraph $(V,S,\ell,o)$.
The free group $F$ generated by the elements of $S$ acts on $V$ as follows:
For every $s \in S$ and $v \in V$ set $v.s = \ell(s,v)$ and $v.s^{-1} = \ell(s^{-1},v)$.
Since the group $F$ is free, this defines an action of $F$ on $V$.
Thus, in order to construct subgroups of $\F_d$, we need to construct some
proper, rooted, $S$-labeled, oriented multigraph where $S$ is a basic set of generators for $\F_d$.

\begin{figure}[h] 
\includegraphics[width=0.7\textwidth]{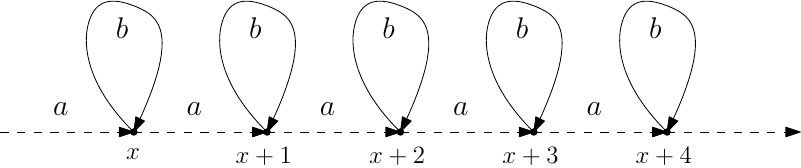}
\caption{Schreier graph $\mathcal{Z}_a$ for $K_a \lhd \F_2 = \IP{a,b}$}
\label{fig:Za}
\end{figure}

\begin{example}
Let $\pi : \F_d \to \Z$ be a surjective homomorphism mapping $\pi(s) =1$
and $\pi\big|_{S \setminus\{s\}} \equiv 0$.
Let $K_s = \ker (\pi)$.  

The Schreier graph of $K_s \lhd \F_d$ is denoted $\mathcal{Z}_s$.
It looks like the graph of $\mathbb{Z}$ with $d-1$ loops at each vertex.  
The loops are labeled by $S \setminus \{s\}$, and 
the oriented edge $(x,x+1)$ is labeled by $s$.
See Figure \ref{fig:Za}.
\end{example}

\subsection{Local properties of Schreier graphs}

Throughout the rest of this section, 
we will always assume some fixed set of free generators $S = \{a_1, \ldots, a_d\}$ for $\F_d$,
used to determine all Schreier graphs.

\begin{definition}
Let $S$ be a finite symmetric generating set for a group $G$.
Let $K \leq G$ be a subgroup, and let $N_1,...,N_m \leq G$ be subgroups. 
We say that $G/K$ is locally-$(G/N_1, \ldots , G/N_m)$ if $\Sch(K)$ satisfies the following:

For every $r>0$ there exists $R>0$ such that for all $Kg \in \Sch(K)$ such that $|Kg|>R$ there exists 
$1 \leq j \leq m$ such that the induced subgraphs on
$B(Kg,r) \subseteq \Sch(K)$ and $B(N_j,r) \subseteq \Sch(N_j)$ 
are isomorphic as rooted, oriented, $S$-labeled multigraphs.
\end{definition}

That is, the Schreier graph of $G/K$ looks locally like one of the Schreier graphs of $G/N_j$, $j=1,\ldots,m$.
Note that this is equivalent to saying that in the Chabauty topology on $\Sub(G)$ the accumulation points of the
orbit $K^G = \{ K^g \ : \ g \in G \}$ are a subset of $\{ N_1, \ldots, N_m \}$.

The following is basically Proposition 4.11 from \cite{HY18}, 
connecting the notion of locality to $\core_\emptyset(K)$.
However, we require something sightly stronger for the proof of Theorem \ref{thm:free groups}, 
which will enable us to intersect subgroups and still preserve local properties.

\begin{lemma}
Let $K , N_1,...,N_m \leq F$ be subgroups.

If $F / K$ is locally-$(F/N_1, \ldots, F/N_m)$ then 
for any $g \in \cap_{j=1}^m N_j$ there exists $R>0$ such that 
if $|Kf| > R$ then $g \in K^f$.
\end{lemma}

\begin{proof}
Let $\Lambda = \Sch(K)$ and $\Gamma_j = \Sch(N_j)$.

Fix $g \in \cap_{j=1}^m N_j$.
Let $r = |g|$.  Let $R>0$ be large enough so that for any $Kf \in \Lambda$ with $|Kf| > R$,
the ball $B_\Lambda(Kf, r)$ is isomorphic to one of $B_{\Gamma_j}(N_j , r)$.

Fix some $Kf \in \Lambda$ with $|Kf|>R$, 
and let $N = N_j , \Gamma = \Gamma_j$ for $j$ such that 
$B_\Lambda(Kf, r)$ is isomorphic to $B_{\Gamma_j}(N_j , r)$.

Write $g = s_1 s_2 \cdots s_{r}$ as a reduced word in the generators of $F$,
and set $g_j = s_1 \cdots s_j$ for all $1 \leq j \leq r$, and $g_0 = 1$.
Since $N g_0 , \ldots, N g_r$ is a path in $\Gamma$, from $N = Ng_0$ to $N g = N g_r$,
of length at most $r$, we have that it is contained in $B_\Gamma(N,r)$.
Thus, it is isomorphic to the path $Kf , Kfg_1 , \ldots, Kf g_r$ in 
$\Lambda$, inside the ball $B_\Lambda(Kg,r)$.  

Since $g \in N$ by assumption, it must be that $N g_r = Ng = N = Ng_0$.
So also in the isomorphic copy of the path we must have that $Kf = Kf g_0 = Kf g_r = Kf g$.
This implies that $g \in K^f$.
\end{proof}

\begin{lemma}
Let $K_1, \ldots, K_m \leq F$ and let $N_{1,1}, \ldots, N_{1, \ell_1} , \ldots,
N_{m,1}, \ldots, N_{m, \ell_m} \leq F$ be subgroups of $F$.
For any $1 \leq j \leq m$ assume that $F/K_j$ is locally-$(F/N_{j,1} , \ldots , F/N_{j,\ell_j})$.

Then, 
$$ \bigcap_{j=1}^m \bigcap_{i=1}^{\ell_j} N_{j,i} \leq \core_\emptyset( K_1 \cap \cdots \cap K_m) . $$
\end{lemma}

\begin{proof}
If $g \in  \bigcap_{j=1}^m \bigcap_{i=1}^{\ell_j} N_{j,i} $,
then for each $j$ there exists some $R_j>0$ such that 
$\{ f \in F \ : \ g \not\in K_j^f \} \subset B_{F/K_j}(K_j,R_j)$.
Thus, for 
$K = K_1 \cap \cdots \cap K_m$ we get that 
\begin{align*}
 || g ||_K & \leq \# \{ f \in F \ : \ g \not\in (K_1 \cap \cdots \cap K_m)^f \} 
\leq \sum_{j=1}^m \# \{ f \in F \  : \ g \not\in K_j^f \} \\
& \leq \sum_{j=1}^m |B_{F/K_j}(1,R_j)| < \infty ,
\end{align*}
so that $g \in \core_\emptyset(K)$.
\end{proof}

We will use the following special case in order to find the subgroups $K_n$ in Theorem \ref{thm:free groups}.

\begin{corollary}  \label{cor:locally nilpotent}
Let $N_1 , \ldots,  N_m \leq F$ be subgroups such that $M \leq N_1 \cap \cdots \cap  N_m$ 
for some subgroup $M \leq F$.
Let $(K_n)_n$ be a sequence of subgroups such that 
$F/K_n$ is locally-$(F/N_1 , \ldots,  F/N_m)$ for every $n$.

Then, $M \leq \core_\emptyset ( \cap_{j=1}^n K_j )$ for any $n$.
\end{corollary}

\subsection{Gluing graphs}

In this subsection we describe how to glue together Schreier graphs, in order to obtain 
a Schreier graph of $K \leq \F_d$ that is ``locally nilpotent''.

~ \\

\begin{definition}
For $1 \neq g \in \F_d$, write $g = s_1 \cdots s_n$ as a minimal word in the generators and their inverses
(so $n = |g|$).  Define
$$ \rad(g) = \max \{ k \geq 0 \ : \ s_1 \cdots s_k = (s_{n-k+1} \cdots s_n)^{-1} \} . $$
Note that $0 \leq \rad(g) < \tfrac{|g|}{2}$.

For $K \leq \F_d$ set $\rad(K) = \min \{ \rad(g) \ : \ 1 \neq g \in K \}$.
\end{definition}

It is immediate that if $K \leq M \leq \F_d$ then $\rad(K) \geq \rad(M)$.

\begin{proposition}
Let $K \leq \F_d$ with $\rad(K) \geq r$.
Then, for any $Kg$ in the Schreier graph $\Sch(K)$ with $|Kg| > r$, there exists 
a unique $g_r \in \F_d$ with $|g_r| = r$ such that 
any path in $\Sch(K)$ from $Kg$ to $K$ must pass through $Kg_r$. 
\end{proposition}

\begin{proof}
Recall that $S$ is the fixed set of generators for $\F_d$.

Let $|Kg| > r$.  
Write $g = s_1 \cdots s_n$ for $n = |g|$ and $s_j \in S \cup S^{-1}$.
Set $g_j = s_1 \cdots s_j$. 

We claim that any path from $K$ to $Kg$ in $\Sch(K)$ must pass through $K g_r$.
Indeed, let 
$K = Kf_0 , Kf_1 , \ldots, K f_m = Kg$ be any simple path in $\Sch(K)$ 
from $K$ to $Kg$.
By passing to equivalent elements ${\pmod K}$, we may assume that 
$f_j = t_1 \cdots t_j$
for some generators and inverses 
$t_1 , \ldots, t_m \in S \cup S^{-1}$.

Consider, first, 
$$ h = s_1 \cdots s_n \cdot t_m^{-1}  \cdots t_1^{-1} = g f_m^{-1} \in K . $$
If $h \neq 1$ then since $\rad(h) \geq \rad(K) \geq r$, 
this is only possible if $s_1 = t_1, \ldots, s_r = t_r$.
Thus, $K f_r = K g_r$, which completes the proof.

The only other possibility is that $h=1$ which implies that $s_j = t_j$ for all $j \leq n=m$.
Since $|Kg| , |Kf_m| >r$ it must be that $n,m >r$, so $Kf_r = Kg_r$, which completes the proof.
\end{proof}

\begin{proposition} \label{prop:large rad(K)}
If $K \leq \F_d$ with $\rad(K) \geq r$, then the ball of radius $r$ in $\Sch(K)$ is isomorphic 
(as a rooted oriented labeled graph) to the ball of radius $r$ in $\F_d$.
\end{proposition}

\begin{proof}
One just checks that the map $f \mapsto Kf$, restricted to $|f| \leq r$, 
is a bijection preserving the edges and labeling.
Indeed, if $Kf = Kg$ for $|f| , |g| \leq r$, then write $f,g$ as reduced words in $\F_d$,
$g = s_1 \cdots s_n$ and $f = t_1 \cdots t_m$ for $s_j , t_j \in S \cup S^{-1}$. 
We have that $g f^{-1} = s_1 \cdots s_n \cdot t_m^{-1} \cdots t_1^{-1} \in K$,
so $\rad(g f^{-1} ) \geq r \geq \max \{ n,m \}$.  Hence it must be that $g=f$.
\end{proof}

\begin{definition}
Let $K \leq \F_d$, and let $r \leq \rad(K)$.
Define the \define{$r$-prefix} of $Kg$, denoted $\pref_r(Kg)$, as the following element of $\F_d$:

If $|Kg| > r$ then define $\pref_r(Kg) = f$ for the unique $f \in \F_d$ with $|f|=r$, such that
any path in $\Sch(K)$ from $Kg$ to $K$ must pass through $Kf$.

If $|Kg| \leq r$ then define $\pref_r(Kg) = f$ for the unique $f \in \F_d$ such that $|f| = |Kg|$
and $Kf = Kg$.
\end{definition}

That is, if $\rad(K) \geq r$, then $\Sch(K)$
looks like graphs attached to the leaves 
of a depth-$r$ regular tree, labeled by the generators $a_1, \ldots ,a_d$ and their inverses.
The $r$-prefix of a vertex 
at distance more than $r$ from the root is given as follows: this vertex 
is in some unique subgraph ``hanging'' off
some vertex $u$ at the $r$-th level of the tree.  
This specific vertex $u$ can be canonically identified with an element $f \in \F_d$ 
such that $|f|=r$, because the ball of radius $r$ in $\Sch(K)$ about the root 
is isomorphic to the ball of radius $r$ in $\F_d$ about the origin.
The element $f$ will be the $r$-prefix.

Note that one may take $K = \{1\}$ in the definition of the $r$-prefix, 
so that $\pref_r(g)$ is defined for elements $g \in \F_d$.
When $|g| > r$,
$\pref_r(g)$ is just the ancestor of $g$ in level $r$ of the $2d$-regular tree which is the 
Cayley graph of $\F_d$.  When $|g| \leq r$ then $\pref_r(g) = g$.

This notion of prefix will be very useful in approximating the boundary behavior of the random walk
on the free group, see below in the proof of Lemma \ref{lem:main lem}.

If $\mu$ is a symmetric, adapted measure on $\F_d$ with finite first moment,
then the $\mu$-random walk $(X_t)_t$ is always transient.
Thus, for any $r>0$, the sequence $(\pref_r(X_t))_t$ stabilizes a.s.\ on some element in $\F_d$.
We define $\pref_r(X_\infty) = \lim_{t \to \infty} \pref_r(X_t)$ to be this a.s.\ limit.
See Corollary \ref{cor:pref at infinity} below for a more general situation.

One main property of $\pref_r(X_\infty)$ is given by the following, 
which is Proposition 14.41 of \cite{LyonsPeres}.

\begin{proposition} \label{prop:prefix gives boundary}
Let $\mu$ be a symmetric, adapted measure on $\F_d$ with finite first moment.
Let $(X_t)_t$ denote the $\mu$-random walk.
Let $\Tt$ denote the tail $\sigma$-algebra of the walk; that is
$$ \Tt = \bigcap_{t} \sigma( X_t, X_{t+1} , \ldots ) . $$

Then, 
$$ \lim_{r \to \infty } H(X_1 \ | \ \pref_r(X_\infty) ) = H(X_1 \ | \ \Tt) . $$
\end{proposition}

Another property is that upon exiting a large ball, if $\rad(K)$ is much larger than the ball's radius,
the $r$-prefix is basically the same as that of the walk on the free group itself.

\begin{lemma} \label{lem:large radius for prefix}
Let $\mu$ be a symmetric, adapted measure on $\F_d$ with finite first moment.
Let $(X_t)_t$ denote the $\mu$-random walk.
Define the stopping time
$$ T_n = \inf \{ t \geq 0 \ : \ |X_t| > n \} $$
(\ie the first time the walk exits the ball of radius $n$).

For any $\eps>0$ there exist $n_0$ such that for all $n \geq n_0$ the following holds.

Let $K \leq \F_d$ be such that $\rad(K) \geq n^3$.
Then, for any $r \leq n$,
$$ \Pr [ \pref_r(X_{T_n}) \neq \pref_r(K X_{T_n}) ] < \eps . $$
\end{lemma}

\begin{proof}
An important observation is that for $g \in \F_d$, 
$K \leq \F_d$ with $\rad(K) \geq |g|$, we have that $\pref_r(K g) = g$.

Let $j = n^3$.
If $\rad(K) \geq j$, note that 
$\pref_r(X_t) = \pref_r(K X_t)$
for all $t < T_j$.
By definition $T_j \geq T_n$, so
\begin{align*}
\Pr [ \pref_r(X_{T_n}) \neq \pref_r(K X_{T_n}) ] & \leq \Pr [ T_n \geq T_j ] 
\leq \Pr [ T_n = T_j \leq m ] + \Pr [ T_n = T_j > m ] .
\end{align*}
Let $U_t = X_{t-1}^{-1} X_t$, denote the i.i.d.-$\mu$ ``jumps'' of the walk.
The event $\{ T_j \leq m \}$ implies that there exists some $t \leq m$ such that $|U_t| \geq \frac{j}{m}$.
By Markov's inequality, 
$$ \Pr [ T_n = T_j \leq m] \leq m \cdot \frac{m \E|X_1| }{j} . $$
Also, it is well known that since $\F_d$ is non-amenable, $\E[T_n] \leq C_\mu n$,
for some constant $C_\mu>0$ independent of $n$.
So
$$ \Pr [ T_n = T_j > m ] \leq \frac{\E[T_n] }{ m } \leq \frac{C_{\mu}n}{m} . $$
Choosing $m = n^{5/4}$ and recalling that $j = n^3$, we arrive at
\begin{align*}
\Pr [ \pref_r(X_{T_n}) \neq \pref_r(K X_{T_n}) ] & \leq \frac{ \E |X_1| }{ \sqrt{n} }  + \frac{ C_\mu }{ n^{1/4} } ,
\end{align*}
which tends to $0$ as $n \to \infty$.
\end{proof}

We now repeat the construction from Section 4.3 in \cite{HY18}, see there for a proof.
This will provide us with a Schreier graph which has $\rad \geq n$ and is also ``locally nilpotent''.

\begin{lemma} \label{lem:gluing}
Let $N \lhd \F_d$, and let $n>0$ be some integer.
For a generator $s \in S$ recall that $K_s = \ker (\pi_s)$ 
where $\pi_s : \F_d \to \Z$ is the surjective homomorphism mapping $\pi_s(s)=1$ and $\pi\big|_{S \setminus\{s\}} \equiv 0$.

There exists $K \leq \F_d$ such that $\rad(K) \geq n$ and also such that 
$\F_d/K$ is 
locally-$(\F_d/N , \F_d/ K_{a_1} , \ldots, \F_d/ K_{a_d} )$.
\end{lemma}

There are many ways to construct subgroups $K$ with the required properties in Lemma \ref{lem:gluing}.
In \cite{HY18} on such construction is provided.  For our purposes we only require the existence of $K$
with these properties, and the specific construction is not important.

Although the proof of Lemma \ref{lem:gluing} 
is in \cite{HY18}, let us give an intuitive description of the construction the appears there.
We will describe the special case of $d=2$ and $\F_d = \IP{ a,b }$ (so $S=\{a,b\}$),
the more general $d \geq 2$ case being a straightforward adaptation.

We start with part of a Schreier graph by taking $T_n$ to be 
the depth-$n$, $4$-regular tree from the Cayley graph of $\F_2$
(\ie $B_n(1)$ in $\F_2$).
This graph $T_n$ has leaves, each leaf has one incoming edge labeled by some element from $\{ a,a^{-1} , b,b^{-1} \}$.

Now, suppose $g$ is some such leaf, and assume that the incoming edge from its ancestor 
in the depth-$n$ regular tree is labeled by $b$.  Thus, to complete this leaf into a Schreier graph,
we need to add $3$ outgoing edges with labels $a,a^{-1}$ and $b$.

For the outgoing edge labeled by $b$, we consider the following labeled graph $\Nn_g$:
The vertices of $\Nn_g$ are a copy of the set of natural numbers $\N$,
specifically the set $\{ 0_g, 1_g, 2_g , \ldots \}$.
For every $n \in \N$ we place an oriented edge from $n_g$ to $(n+1)_g$, labeled by $b$.
We also add a self loop labeled $a$ to each vertex $n_g$.  
Finally, we connect $\Nn_g$ to the finite tree $T_n$ by adding an oriented edge from $g$ to $0_g$ 
labeled by $b$.

For the pair of outgoing edges from $g$ labeled $a,a^{-1}$ we proceed differently:
Here we take a copy of the Cayley graph of $\F_2 / N$, called $\Lambda_g$.
In this graph there is an oriented edge $(N,Na)$ labeled by $a$.
We remove this edge from the graph, and connect what remains to the finite tree $T_n$
by adding to oriented edges: one from our leaf $g$ to $Na \in \Lambda_g$ labeled $a$,
and another from $N \in \Lambda_g$ to the leaf $g$ also labeled $a$ (so the edge in the reverse orientation from
$g$ to $N$ is labeled $a^{-1}$).
See Figure \ref{fig:one leaf}.

\begin{figure}[h]
\label{fig:one leaf}
\includegraphics[width=0.8\textwidth]{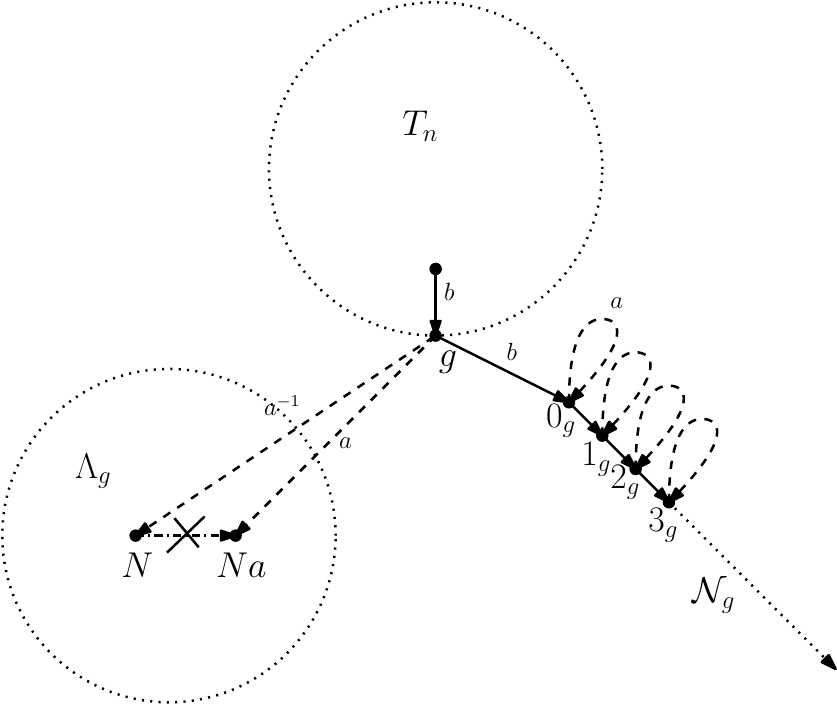}
\caption{Connecting a leaf $g$ of the finite tree $T_n$ to $\Lambda_g$ and $\Nn_g$.  
Solid arrows indicate oriented edges labeled by $b$
and dashed arrows indicted oriented edges labeled by $a$.  Note that the incoming edge to $g$
is labeled by $b$.  The outgoing edges from $g$ are labeled by $a,a^{-1},b$.
The edge from $N$ to $Na$ is removed.}
\end{figure}

One readily sees that any vertex in $\Nn_g$ and the copy of $\F_2/N$ connected to $g$
has the proper number of outgoing edges labeled by the generators $\{a,a^{-1},b,b^{-1}\}$.

Now, suppose that $g'$ is a different leaf of $T_n$ with the edge incoming from its ancestor labeled 
by $b^{-1}$.  We perform he same construction as above, except that the  
oriented edges that were labeled $b$ in $\Nn_g$
are now labeled by $b^{-1}$ in $\Nn_{g'}$.

Similarly, for a leaf $\gamma$ of $T_n$ with incoming edge labeled $s \in \{a,a^{-1} \}$, 
we connect a copy of $\N$ called $\Nn_\gamma$ whose loops are labeled by $b$
and non-loops by $s$, and we also connect $\Lambda_\gamma$, a copy of $\Sch(\F_2/N)$,
where the oriented edge $(N,Nb)$ is removed and oriented edges labeled $b$ from $N$ to $\gamma$
and from $\gamma$ to $Nb$ are added.

Performing these modifications for all leaves of $T_n$, we arrive at a Schreier graph of $\F_2$.
This graph looks like a finite tree, with copies of either $\N$ or $\F_2/N$ hanging of each leaf,
in such a way that copies of $\N$ are connected to a leaf by one edge, and copies of $\F_2/N$ 
are connected to a leaf by a pair of edges replacing one edge removed from the original copy of $\F_2/N$.
See Figure \ref{fig:gluing graph}.

\begin{figure}[h]
\label{fig:gluing graph}
\includegraphics[width=0.5\textwidth]{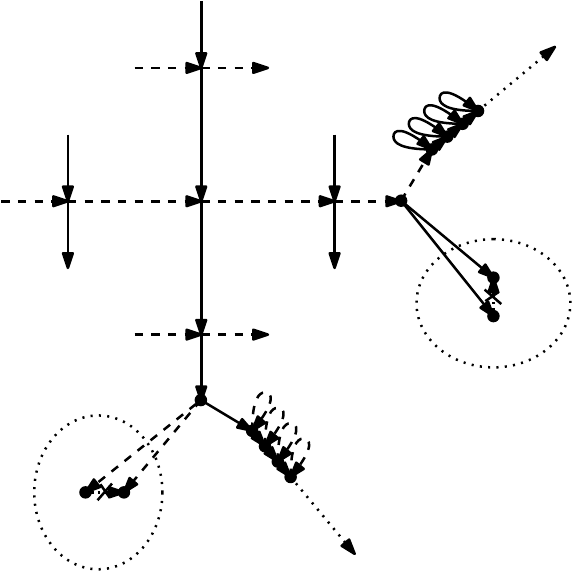}
\caption{Part of the Schreier graph for $n=2$ and $d=2$ in Lemma \ref{lem:gluing},
with the two leaves connected to the corresponding copies of $\N$ and $\F_2/N$.}
\end{figure}

It is not difficult to see that the subgroup $K$ corresponding to this Schreier graph 
has $\rad(K) \geq n$, and also that except for finitely many vertices, the balls in this Schreier 
graph are either balls in $\N$ or balls in $\F_2/N$.

%

\subsection{Transient subgroups}

Let $(X_t)_t$ be a random walk on $\F_d$ with law $\mu$.
Let $K \leq \F_d$ and consider the induced random walk $(K X_t)_t$
on $\F_d/ K$.  We say that $K$ is a \define{$\mu$-transient subgroup}
if $(K X_t)_t$ is a transient Markov chain.  That is, if 
$\Pr [ X_t \in K \ i.o. ] = 0$.
When $\mu$ is a symmetric random walk, the induced walk on $\F_d/K$ is 
a reversible Markov chain.  
One may then prove (see \eg
Chapter 3 of \cite{Woess}) that for the class of 
symmetric random walks with finite second moment,
either $K$ is a transient subgroup for all of them, or for none of them.

\begin{definition}
A subgroup $K \leq \F_d$ is called a \define{transient subgroup} if it is a $\mu$-transient subgroup for
some (and hence any) symmetric random walk $\mu$ on $\F_d$ with finite second moment.
\end{definition}

It is immediate that if $K \leq M \leq \F_d$ and $M$ is a transient subgroup, then also $K$ 
is a transient subgroup.

\begin{lemma}
\label{lem:transient}
Let $N \lhd \F_d$,
and consider the construction from Lemma \ref{lem:gluing} of
$K \leq \F_d$ such that $\rad(K) \geq n$ and $\F_d/K$ is 
locally-$(\F_d/N , \F_d/ K_{a_1} , \ldots, \F_d/ K_{a_d} )$.

If $N$ is a transient subgroup, then
$K$ is a transient subgroup as well.
\end{lemma}

\begin{proof}
Transience of a subgroup $K$ is equivalent to the induced simple random walk on the 
Schreier graph $\Sch(K)$ being transient.
In the construction of Lemma \ref{lem:gluing}, 
the Schreier graph of $K$ is obtained from the graph $\Sch(N)$ by removing one edge
from $\Sch(N)$, and adding a lot of other edges and vertices.
Specifically, embedded in $\Sch(K)$, 
there is a copy of $\Sch(N)$ with one edge removed and a path of length $2$
added in place of that edge.
Network theory (see \eg \cite[Chapter 2]{Woess} or \cite{LyonsPeres}) tells us that since this (modified)
copy of $\Sch(N)$ is transient, also $\Sch(K)$ is transient.
\end{proof}

It is very well known (see \eg \cite{LPW, LyonsPeres, Woess}) that a Markov chain is transient 
if and only if the corresponding {\em Green function} converges.
Specifically, given a $\mu$-random walk $(X_t)_t$ on $\F_d$, for a subgroup $K \leq \F_d$ we write 
$$ \gr_K^{m+}(\alpha,\beta) = \sum_{t=m}^\infty \Pr [ K X_t = \alpha \ | \ K X_0 = \beta ]  $$
for an integer $m \geq 0$.  
By the notation $\Pr [ K X_t = \alpha \ | \ K X_0 = \beta ]$ we refer 
to $\Pr [ K X_t = \alpha \ | \  X_0 = g ] = \Pr_g [ K X_t = \alpha ]$ for some $g \in G$ such that $Kg = \beta$.
One may check that this does not depend on the specific choice of representative $g$.

We also set $\gr_K = \gr_K^{0+}$.
As mentioned, 
$\gr_K(\alpha,\beta) < \infty$ if and only if $K$ is a $\mu$-transient subgroup
(finiteness does not depend on the specific choice of $\alpha,\beta$, although the actual value
of $\gr_K(\alpha,\beta)$ does).
We also use the notation 
$$ \gr_K^{m+}(\alpha,B) = \sum_{\beta \in B} \gr_K^{m+}(\alpha,\beta)  $$
for a subset $B \subset G/K$.

Let us now record some properties of Green functions on the Schreier (Cayley) graphs of the 
normal subgroups of $\F_d$.

Let $(X_t)_t$ denote the $\mu$-random walk, for 
$\mu$ an adapted, symmetric measure on $\F_d$,
with finite first moment.  
Let $N \lhd C \lhd \F_d$ be normal subgroups.

First note that since $N \lhd C$,
we have that $\Pr [ N x X_t = N y ] \leq \Pr [ C x X_t = C y ]$, so
\begin{align*}
\gr_{N}^{m+}(N x , N y) & = \sum_{t \geq m} \Pr [ N x X_t = N y ] 
\leq \gr_{C}^{m+} ( C x , C y) .
\end{align*}
Using the fact that $\F_d / C$ is a group, it is well known  that
\begin{align*}
\Pr [ C x X_{2t} = C y ] & \leq \sqrt{ \Pr [ C x X_{2t} = C x  ] \cdot 
\Pr [ C y X_{2t} = C y ] }  = \Pr [ C X_{2t} = C ] \\
\Pr [ Cx X_{2t+1} = Cy ] & \leq \sqrt{ \Pr [ C x X_{2t+2} = C x  ] \cdot
\Pr [ C y X_{2t} = C y ] } \leq \Pr [ C X_{2t} = C ] \\
\end{align*}
(see \eg \cite{Woess}, this is a simple Cauchy-Schwarz application),
so that
\begin{align} \label{eqn:CS for Green}
\gr_{C}^{m^+} (C x , C y) & 
\leq 2 \sum_{t \geq m-1} \Pr [ X_{2t} \in C ] 
\leq 2 \gr^{(m-1)+} (1,C) , 
\end{align}
where $\gr$ is the Green function of the $\mu$-random walk $(X_t)_t$ on $\F_d$.

Also, 
for any $k>0$,
denote $A_N  = \{ y \in \F_d \ | \ |N y| \leq k \}$ and $A_C = \{ y \in \F_d \ | \ |Cy| \leq k \}$
(which is the inverse image in the group $\F_d$, of the ball of radius $k$ around the root in $\Sch(K), \Sch(C)$,
respectively).
Since $N \lhd C$, we have that $A_N \subset A_C$.
Thus, 
\begin{align} \label{eqn:monotone for balls}
\gr_{N}^{m+}(N x , B_{\F_d/N}(N, k) ) & = \gr^{m+} (x , A_N) \leq \gr^{m+} (x, A_C) = 
\gr_{C}^{m+} (C x, B_{\F_d/C}(C, k) ) . 
\end{align}

\subsection{Full entropy approximation}

We now move to the proof of Theorem \ref{thm:free groups}.
The final step is in the following lemma.

\begin{lemma} \label{lem:main lem}
Let $(C_n)_n$ be a non-increasing sequence $C_{n+1} \lhd C_n \lhd \F_d$ of normal subgroups
such that $\rad(C_n) \geq n$ for each $n$.
Let $\mu$ be a symmetric random walk on $\F_d$
with finite second moment, and assume that $C_1$
is a $\mu$-transient subgroup.  
Then, 
$$ \lim_{n \to \infty} h(\F_d / C_n , \bar \mu ) = h( \F_d , \mu) . $$
\end{lemma}

The proof of this lemma is in Section \ref{scn:proof of main lem}.

\begin{remark}
In light of this lemma, 
it is natural to ask whether it suffices to have a non-increasing sequence 
$C_{n+1} \lhd C_n \lhd \F_d$ such that $\bigcap_n C_n = \{1\}$ in order to deduce that 
$h(\F_d / C_n , \bar \mu) \to h( \F_d , \mu)$.  

However, this is not the case, since $\F_d$ is residually finite, which just means that we can find 
$C_{n+1} \lhd C_n \lhd \F_d$ and $\bigcap_n C_n = \{1\}$ as above with $[\F_d : C_n ] <\infty$ for all $n$. 

So in order to get the entropies $h(\F_d / C_n , \bar \mu)$ approaching the full entropy $h(\F_d,\mu)$,
we require some kind of stronger condition on the $C_n$'s becoming small.
One such condition is exactly the content of Lemma \ref{lem:main lem}.
\end{remark}

With Lemma \ref{lem:main lem}, we are ready to prove Theorem \ref{thm:free groups}.

\begin{proof}[Proof of Theorem \ref{thm:free groups}]
Let $N = [\F_d , [\F_d,\F_d] ]$.  So $\F_d / N$ is nilpotent, but not virtually Abelian.
For a generator $s \in \{a_1, \ldots, a_d\}$ of $\F_d$, let $K_s = \ker(\pi_s)$ 
where $\pi_s : \F_d \to \Z$ is the surjective homomorphism mapping $\pi_s(s) = 1$.

For each $n$ use Lemma \ref{lem:gluing} to construct a subgroup $K_n' \leq \F_d$
such that $\rad(K_n') \geq n$, and such that $\F_d / K'_n$ is 
locally-$(\F_d / N , \F_d / K_{a_1} , \ldots, \F_d/  K_{a_d})$.
Set $K_n = \bigcap_{j \leq n} K_j'$.  
(We do this because we require a non-increasing sequence of subgroups
to apply Lemma \ref{lem:main lem}.)
By Corollary \ref{cor:locally nilpotent},  we have that 
$N \lhd \core_\emptyset(K_n)$ for all $n$, implying that 
$\F_d / \core_\emptyset(K_n)$ is (at most $2$-step) nilpotent.
By the Choquet-Deny Theorem \cite{CD, Raugi}, $h(\F_d/ \core_\emptyset(K_n) , \bar \mu) = 0$,
so $p \mapsto h(\F_d , \mu, \lambda_{p,K_n} )$ is a continuous function by 
Proposition \ref{prop:cont entropy}.

Note that since $\F_d/ N$ is not virtually Abelian, $N$ cannot be a recurrent subgroup,
for any symmetric random walk $\mu$ with finite second moment
(see \eg \cite[Chapter 3.B]{Woess}).  Thus, the construction in Lemma \ref{lem:gluing}
guaranties that $K_1'$ is also a transient subgroup, by Lemma \ref{lem:transient}.
Since $\core_{\F_d}(K_1) \lhd  K_1 = K'_1$, we can apply Lemma \ref{lem:main lem}
to the sequence $C_n = \core_{\F_d} (K_n)$, to obtain that 
$$ \lim_{n \to \infty} h( \F_d / C_n , \bar \mu ) = h( \F_d , \mu ) . $$
This completes the proof.
\end{proof}

\section{Proof of Lemma \ref{lem:main lem}}

\label{scn:proof of main lem}

The following lemma's proof was inspired by the work of Stankov \cite{Stankov}.

\begin{lemma} \label{lem:Stankov}
Let $G$ be a finitely generated group, with some fixed finite generating set $S$.
Let $K \leq G$ be a subgroup and consider $\Gamma = \Sch(K)$, the Schreier graph of $K$
(with respect to the generating set $S$).
Let $\mu$ be a symmetric, adapted measure on $G$ with finite first moment,
and let $(X_t)_t$ denote the corresponding random walk.

Let $A \subset \Gamma$ be some finite subset, and consider the connected components of 
$\Gamma \setminus A$ in the Schreier graph.
Let $r(A) = \max \{ |\beta| \ : \ \beta \in A \}$ 
(recalling that $|\beta|$ is the distance from $\beta$ to the root in $\Gamma$).

For $\alpha \in \Gamma$ let $\Ee_t(\alpha)$ denote the event that $\alpha X_t$ and $\alpha X_{t+1}$ are 
both not in $A$, and are each in different connected components of $\Gamma \setminus A$.

Then, for any integer $m \geq 0$,
$$ \sum_{t \geq m} \Pr [ \Ee_t(\alpha) ] \leq |A| \cdot \E \big[ |X_1| \cdot 
\gr_K^{m+}(\alpha, B(K, r(A)+|X_1| ) ) \big] $$
(where $B(K,r)$ is the ball of radius $r$ about $K$ in $\Gamma$).
\end{lemma}

\begin{proof}
For any $u \in G$ choose some generators (and inverses)   $s_j(u) \in S \cup S^{-1}$ 
so that $u=s_1(u)s_2(u) \cdots s_{|u|}(u)$.
Also denote $u_0=1$ and $u_j=s_1(u) \cdots s_j(u)$ for all $j=1,\ldots ,|u|$. 
Set
$$ Y(\alpha,u)=\sum_{j=1}^{|u|} \1{ \alpha u_j \in A } . $$
Write $U_t = X_{t-1}^{-1} X_t$, so that $(U_t)_{t \geq 1}$ are i.i.d.-$\mu$ elements.
Note that 
$$ \Ee_t(\alpha) \subseteq \{Y(\alpha X_t, U_{t+1}) \geq 1\} . $$
Hence by Markov's inequality,
$$ \Pr [ \Ee_t(\alpha)) ] \leq \E [ Y(\alpha X_t,U_{t+1}) ] \leq 
\sum_{u \in G} \mu(u) \sum_{j=1}^{|u|} \sum_{\beta \in A} \Pr [ \alpha X_t=\beta u_j^{-1} ] , $$
and thus,
\begin{align*}
\sum_{t \geq m} \Pr [ \Ee_t(\alpha) ] & \leq \sum_{u \in G} \mu(u) \cdot \sum_{\beta \in A} \sum_{j=1}^{|u|} 
\gr_K^{m+}(\alpha,\beta u_j^{-1}) \\
& \leq |A| \cdot \sum_{u \in G} \mu(u)  |u| \cdot \sum_{|\beta| \leq r(A)+|u|} \gr_K^{m+} (\alpha,\beta)
\\
& = |A| \cdot \E \big[ |X_1| \cdot \gr_K^{m+}(\alpha, B(K, r(A)+|X_1| ) ) \big] 
\end{align*}
\end{proof}

Recall the definition of $\pref_r$ for Schreier graphs of the free group.

\begin{corollary} \label{cor:change prefix}
Let $K \leq \F_d$ with $\rad(K) \geq n$.  Let $\Gamma = \Sch(K)$ be the Schreier graph of $K$.
Let $B_r = B(K,r)$ be the ball of radius $r$ about the root in $\Gamma$.

Let $Z_t = K X_t$, where $(X_t)_t$ is a $\mu$-random walk, for some symmetric, adapted measure $\mu$ 
on $\F_d$, with finite first moment.

Then, for every $r \leq n$ and any $m \geq 0$,
$$ \sum_{t \geq m} \Pr \big[ 
\pref_r(Z_t) \neq \pref_r(Z_{t+1}) \big]
\leq |B_r| \cdot \E \big[ |X_1| \cdot \gr_K^{m+}(K, B_{r+|X_1|} ) \big] . $$
\end{corollary}

\begin{proof}
Note that if $\pref_r(Z_t) \neq \pref_r(Z_{t+1})$ then $Z_t, Z_{t+1}$ must be in different components of 
$\Gamma \setminus B_r$. This is because $\rad(K) \geq n \geq r$. 
Taking $\alpha = K$ in Lemma \ref{lem:Stankov} completes the proof.
\end{proof}

\begin{corollary}
\label{cor:pref at infinity}
Let $K \leq \F_d$ with $\rad(K) \geq n$.  Let $\Gamma = \Sch(K)$ be the Schreier graph of $K$.
Let $(X_t)_t$ denote a $\mu$-random walk, for some symmetric, adapted measure $\mu$ 
on $\F_d$, with finite first moment.
Assume that $K$ is $\mu$-transient.

Then, for any $r \leq n$ we have that 
$\pref_r(K X_\infty) : = \lim_{t \to \infty} \pref_r(K X_t)$ is well defined (as an element of $B_{\F_d}(1,r)$).
\end{corollary}

\begin{proof}
Since $K$ is $\mu$-transient, 
$\E \gr_K(K, B_{r+|X_1|}) < \infty$,
so by the previous corollary and Borel-Cantelli 
the sequence of elements  $(\pref_r(K X_t) )_t$ stabilizes eventually a.s.
\end{proof}

\begin{lemma} \label{lem:speed}
Let $C \lhd \F_d$ with $\rad(C) \geq 1$.
Assume that $C$ is $\mu$-transient,  for some symmetric, adapted measure $\mu$ 
on $\F_d$, with finite first moment.  

Then, there exists a constant $M = M_\mu >0$ such that for all $r>0$,
$$ \gr_{C} (Cx, B_{\F_d/C}(C,r) ) \leq M r .   $$
\end{lemma}

\begin{proof}
Since $\rad(C) \geq 1$, by Proposition \ref{prop:large rad(K)} 
the Cayley graph of $\F_d/C$ (which is $\Sch(C)$) has at least $2d$ topological ends.
Stalling's Theoreom (see \cite[Section 2]{DrutuKap} or \cite[Chapter 3]{GaborBook}) implies that 
the group $\F_d/C$ (is a free product amalgamated over a finite subgroup, and therefore) is non-amenable.
Thus, by Kesten's Amenability Criterion (see \eg \cite[Chapter 6]{LyonsPeres} or \cite[Chapter 7]{GaborBook}),
we know that there exists $\eps>0$ such that for all $t$ and all $x,y \in G$ we have 
$$ \Pr_{Cx} [ C X_t = Cy ] \leq e^{-\eps t}  $$
where $(X_t)_t$ is the $\mu$-random walk.
Thus, by Borell-Cantelli, we have that 
$$ \liminf_{t \to \infty} \frac1t |C X_t| > 0 \qquad a.s. $$

The proof now follows exactly as in Lemma 2.1 of \cite{LPSZ}.
Since it is very short we provide it here for completeness.

We may choose $\delta>0$ and $t_0$ so that for all $t \geq t_0$ we have 
$$ \Pr_{C} [ \forall \ t \geq t_0 \ , \ |CX_t| > \delta t ] > \frac12 . $$
Let $M = \max \{ 2 \lceil \tfrac{1}{\delta} \rceil , t_0 \}$.
Then, for any $t \geq 0$,
\begin{align*}
\Pr_{Cx}  [ \forall \ s & \geq t+Mr \ , \ |CX_s| > r \ | \ |CX_t| \leq r ] 
 \geq \inf_{|Cy| \leq r } \Pr_{Cy} [ \forall \ s \geq Mr \ , \ |CX_s| > r ] \\
& \geq \Pr_C [ \forall \ s \geq Mr \ , \ |CX_s| > \delta s ]
> \frac12
\end{align*}
(because $s \geq Mr$ implies that $\delta s \geq 2r$).
The Markov property tells us that for all $k \geq 1$,
$$ \Pr_{Cx} [ \exists \ s > k Mr \ , \ |CX_s| \leq r ] < 2^{-k} , $$
implying that 
$$ \gr_{C} (Cx, B_{\F_d/C}(C,r) ) \leq 2 M r . $$
\end{proof}

The next lemma is our final estimate before proving Lemma \ref{lem:main lem}.
It basically quantifies the fact that when $\rad(K) \geq n \geq r$, 
we can approximate $\pref_r(K X_\infty)$ by $\pref_r(K X_{T_n})$. 
Here $T_n$ is the stopping time
$$ T_n = \inf \{ t \geq 0 \ : \ |X_t| > n \} . $$
This is the first time the walk $(X_t)_t$ exists the ball of radius $n$ in $\F_d$.
The main idea behind the proof is that if $n$ is large enough, then once the walk $(X_t)_t$ 
exits the ball of radius $n$, it is very difficult for it to return to the ball of radius $r$.
However, in order to change the $r$-prefix, the walk must return to the ball of radius $r$,
so by choosing $n$ large enough we can make the probability of this event arbitrarily small.
If the random walk had finite support, this would be straightforward enough.
However, an extra technical difficulty arises when the walk can ``jump'' arbitrarily far. 
One needs to somehow control the behavior of the walk preventing it from jumping so far that it 
``passes through'' the ball of radius $r$ during the ``jump'', changing the prefix.

\begin{lemma} \label{lem:Zinfinity and ZTn}
Let $(C_j)_j$ be a non-increasing sequence of normal subgroups of $\F_d$ (so $C_{j+1} \lhd C_j \lhd \F_d$).
Assume that $\rad(C_j) \geq j$ for all $j$. 
Let $\mu$ be a symmetric, adapted measure on $\F_d$ with finite second moment,
and let $(X_t)_t$ denote the corresponding random walk.
Assume that $C_1$ is $\mu$-transient.

Then,
for any integer $r >0$ and any $\eps>0$, there exists $n_0$ 
such that for all $n \geq n_0$ and all $j > r$,
$$ \Pr \big[ \pref_r(C_j X_\infty) \neq \pref_r ( C_j X_{T_n} ) \big] < \eps . $$
\end{lemma}

\begin{proof}
If we denote $B_k = B_{\F_d}(1,k)$ the ball of radius $k$ in $\F_d$,
then $|B_{\F_d/C_j} (C_j , k) | \leq |B_k|$.
Averaging over the values of $T_n$, and using \eqref{eqn:CS for Green}, for any $k>0$,
\begin{align} \label{eqn:for small R}
\E \gr_{C_1}^{T_n+} (C_1 , B_{\F_d/ C_1} (C_1, k ) ) & 
\leq 2 | B_{\F_d/C_1} (C_1 , k) | \cdot \E \gr^{(T_n-1)+} (1,C_1) \nonumber \\
& \leq 2 |B_k| \cdot \E [ \gr^{(T_n-1)+} (1,C_1) \cdot (\1{ T_n-1 \geq m} + \1{T_n-1 < m} ) ] \nonumber \\
& \leq 2|B_k| \cdot \gr^{m+}_{C_1} (C_1,C_1) + 2|B_k| \cdot \Pr [ T_n \leq m ] \cdot \gr_{C_1}(C_1,C_1) .
\end{align}

Also, 
applying \eqref{eqn:monotone for balls},
and using 
Lemma \ref{lem:speed}, we can deduce that
\begin{align} \label{eqn:for big R}
\gr_{C_j}^{m+}(C_j x , B_{\F_d/C_j}(C_j , k) ) \leq \gr_{C_1}^{m+} (C_1 x, B_{\F_d/C_1}(C_1, k) ) \leq M \cdot k ,
\end{align}
for some constant $M>0$ depending only on $\F_d, \mu , C_1$,
and all $m,k>0$.

Now, fix $r>0$ and $\eps>0$. 
Choose $R>0$ large enough so that 
$$ 2|B_r| \cdot \sum_{|u| > R} \mu(u) |u| \cdot M (r+|u|) < \frac{\eps}{3} , $$
using the fact that $\mu$ has finite second moment.
We know that $C_1$ is $\mu$-transient.
Hence, we can choose $m$ large enough so that 
$$ 2|B_r| \cdot R \cdot |B_{r+R}| \cdot \gr^{m+}_{C_1} (C_1, C_1) < \frac{\eps}{3} . $$ 
Finally, since
$$ \Pr [ T_n \leq m ] \leq \Pr [ \exists \ k \leq m \ , \ |X_{k-1}^{-1} X_k | > \tfrac{n}{m} ]
\leq \frac{ m^2  \E|X_1| }{n } , $$
we may choose $n_0$ large enough so that for all $n \geq n_0$ we have 
$$ 2 |B_r| \cdot R \cdot |B_{r+R}| \cdot \Pr [  T_n \leq m ] \cdot \gr_{C_1}(C_1,C_1) < \frac{\eps}{3} . $$
Combining all of these into \eqref{eqn:for small R} and \eqref{eqn:for big R}, we conclude
\begin{align} \label{eqn:eps bound}
2 |B_r| \cdot &  \sum_{u \in \F_d} \mu(u) |u| \cdot \E \gr_{C_j}^{T_n+} (C_j , B_{\F_d/C_j} (C_j, r+|u| ) ) 
\nonumber \\
& \leq 2|B_r| \cdot R \cdot \E \gr_{C_1}^{T_n+} (C_1 , B_{\F_d/C_1} (C_1, r+R) )
+ 2|B_r| \cdot \sum_{|u| > R} \mu(u) |u| \cdot M (r+|u|)  < \eps .
\end{align}

Now, let $j>r$.
Note that if $\pref_r(C_j X_\infty) \neq \pref_r(C_j X_{T_n})$, then there must exist $t \geq T_n$
such that $\pref_r(C_j X_t) \neq \pref_r(C_j X_{t+1})$.
Corollary \ref{cor:change prefix} tells us that conditional on $T_n=m$,
\begin{align} \label{eqn:pref change after Tn}
\Pr [ \exists \ t \geq T_n  & \ , \  \pref_r(C_j X_t) \neq \pref_r(C_j X_{t+1}) \ | \ \Ff_{T_n} \ , \ T_n= m ] 
\nonumber \\
& \leq |B_{\F_d/C_j}(C_j,r)| \cdot \E \big[ |X_1| \cdot \gr_{C_j}^{m+} (C_j , B_{\F_d/C_j}(C_j , r+|X_1|) ) \big] 
\nonumber \\
& \leq |B_r| \cdot \sum_{u \in \F_d} \mu(u) |u| \gr_{C_1}^{m+} ( C_1 , B_{\F_d/C_1}(C_1, r+|u| ) ) .
\end{align}
Averaging over $T_n$, and plugging this into \eqref{eqn:eps bound}, we conclude that
for any $n \geq n_0$ and $j > r$, 
$$ \Pr [ \pref_r(C_j X_\infty) \neq \pref_r(C_j X_{T_n}) ] < \eps . $$
\end{proof}

We culminate this section with the proof of Lemma \ref{lem:main lem}.

\begin{proof}[Proof of Lemma \ref{lem:main lem}]
Let $(X_t)_t$ denote the $\mu$-random walk on $\F_d$.
Recall that 
$$ \mu^t(C_1 x) = \Pr [ C_1 X_t = C_1 x] = \Pr [ X_t \in C_1 x ] . $$
A simple Cauchy-Schwarz argument shows that $\mu^{t} (C_1 x) \leq \sqrt{ \mu^{2t}(C_1) }$,
which tends to $0$ as $t \to \infty$ (because $\F_d / C_1$ is infinite and $\mu$ is adapted).
It will be convenient below to have $\sup_x  \mu^t(C_1 x) < e^{-1}$,
so fix once and for all some large enough $t>0$ so that this holds.

Let $\Tt$ be the tail $\sigma$-algebra of $(X_t)_t$ and let 
$\Tt_j$ be the tail $\sigma$-algebra of the process $(C_j X_t)_t$.
By \cite{KV83} we know that
\begin{align*}
t \cdot h(\F_d, \mu) & = H(X_t) - H(X_t \ | \ \Tt) \\
t \cdot h(\F_d/C_j , \bar \mu) & = H(C_j X_t) - H(C_j X_t \ | \ \Tt_j ) .
\end{align*}
We have already seen that $h(\F_d/C_j , \bar \mu) \leq h(\F_d , \mu)$,
so we only need to bound the limit in the other direction.

We begin with an upper bound on $H(C_j X_t \ | \ \Tt_j )$.
Note that $\pref_r(C_j X_\infty)$ is measurable with respect to $\Tt_j$. 
Since $C_j X_t$ is a function of $X_t$, we have
\begin{align*}
H(C_j X_t \ | \ \Tt_j) & \leq H(C_j X_t \ | \ \pref_r(C_j X_\infty) ) \leq H(X_t \ |  \ \pref_r(C_j X_\infty) ) \\
& \leq H(X_t \ | \ \pref_r(X_\infty) ) ) + H( \pref_r(X_\infty) \ | \ \pref_r(C_j X_\infty) ) .
\end{align*}

Fix some $\eps>0$.
Using Proposition \ref{prop:prefix gives boundary} (with the random walk induced by $\mu^t$),
choose $r_0$ be large enough so that for all $r>r_0$ we have
$H(X_t \ | \ \Tt) \leq H(X_t \ | \ \pref_r(X_\infty) ) + \eps$.
Then, for all $r>r_0$,
\begin{align*} 
H(C_j X_t \ | \ \Tt_j) & \leq H(X_t \ | \ \Tt) + H( \pref_r(X_\infty) \ | \ \pref_r(C_j X_\infty) ) +\eps .
\end{align*}

We move to bound $H( \pref_r(X_\infty) \ | \ \pref_r(C_j X_\infty) )$.
Recall the stopping time $T_n = \inf \{ s \ : \ |X_s| > n \}$.
Then we can write:
$$ H( \pref_r(X_\infty) \ | \ \pref_r(C_j X_\infty) ) \leq \alpha + \beta_n + \gamma_n , $$
where
\begin{align*}
\alpha & = H( \pref_r(X_\infty) \ | \ \pref_r(X_{T_n} ) ) \\
\beta_n & = H( \pref_r(C_{n^3} X_{T_n}) \ | \ \pref_r(C_{n^3} X_{\infty} ) ) \\
\gamma_n & = H( \pref_r(X_{T_n}) \ | \ \pref_r(C_{n^3} X_{T_n} ) ) .
\end{align*}
Now, note that $\pref_r$ always takes values in $B_r = B_{\F_d}(1,r)$, 
the ball of radius $r$ about the unit element in $\F_d$,
which is specifically a finite set.
Fano's inequality (see \eg \cite{CoverThomas}) tells us that for random variables $X,Y$
taking values in a finite set $F$,
$$ H(X \ | \ Y) \leq H(p,1-p) + p \log |F| \qquad \textrm{ where } p = \Pr [ X \neq Y ] . $$
(Recall that $H(p,1-p) = - p \log p - (1-p) \log (1-p)$.)
Since $H(p,1-p) \to 0$ as $p \to 0$, for our purposes it will suffice to bound 
\begin{align*}
\alpha' & = \Pr [ \pref_r(X_\infty) \neq \pref_r(X_{T_n} ) ] \\
\beta_n' & = \Pr [ \pref_r(C_{n^3} X_{T_n}) \neq \pref_r(C_{n^3} X_{\infty} ) ] \\
\gamma_n' & = \Pr [ \pref_r(X_{T_n}) \neq \pref_r(C_{n^3} X_{T_n} ) ] .
\end{align*}
Lemma \ref{lem:Zinfinity and ZTn} tells us that for any $r>0$ there exists $n_0$ such that for all $n>n_0$
and $j>r$ we have that $\alpha' < \frac{\eps}{ \log |B_r| }$ and $\beta_{n}' < \frac{\eps}{ \log |B_r| }$.
By Lemma \ref{lem:large radius for prefix}, we may choose $n_0$ large enough
so that as long as $n > n_0$ we have that $\gamma_n' < \frac{\eps}{ \log |B_r| }$.
We obtain that for $\eta \in \{\alpha , \beta_n , \gamma_n \}$ we have
$$ \eta < H(\eta' , 1-\eta') + \eta' \cdot \log |B_r| < H(\eps , 1-\eps) + \eps . $$
We conclude:
For any $\eps>0$ 
there exists $j_0$, such that for all $j>j_0$ we have 
$$ H(C_j X_t \ | \ \Tt_j) \leq H(X_t \ | \ \Tt) + 3 H(\eps,1-\eps) + 4\eps . $$
This concludes the upper bound on $H(C_j X_t \ | \ \Tt_j )$.

So we are only left with proving that  $\lim_{j \to \infty} H(C_j X_t) = H(X_t)$.
To this end, given $\eps>0$, choose $j_0$ large enough so that 
$$ - \sum_{|u| > j_0 } \mu^t(u) \log \mu^t(u) < \eps , $$
which can be done since $H(X_t) < \infty$.
Since $C_j \lhd C_1$ we have that $\mu^t(x) \leq \mu^t(C_j x) \leq \mu^t(C_1 x)  <  e^{-1}$,
by our initial assumption on $t$.  
The function $\xi \mapsto - \xi \log \xi$ is increasing as long as $\xi < e^{-1}$.
So $- \mu^t(C_j x) \log \mu^t(C_j x) \geq -\mu^t(x) \log \mu^t(x)$.
Also, since $\rad(C_j) \geq j$, we know that if $|C_j x| \leq j$ then $C_j x = \{x\}$ for $|x| \leq j$.
Hence, we may bound:
\begin{align*}
H(C_j X_t) & = - \sum_{C_j x \in \F_d/C_j } \mu^t(C_j x) \log \mu^t(C_j x) \geq
- \sum_{\substack{ |C_j x| \leq j \\ C_j x \in \F_d/C_j } } \mu^t(C_j x) \log \mu^t(C_j x)  \\
& \geq - \sum_{|x| \leq j} \mu^t (x) \log \mu^t(x) > H(X_t) - \eps ,
\end{align*}
where the last inequality holds as long as $j > j_0$.
\end{proof}

%
%
%

\end{document}